\documentclass[11pt,a4paper]{amsart}

\usepackage{graphicx}
\usepackage{enumerate}
\usepackage{amsmath, amssymb, amsthm, amscd}
\usepackage{mathtools}
\usepackage[dvipsnames]{xcolor}
\usepackage{caption}
\usepackage{subcaption}
\usepackage{multirow}
\usepackage{extarrows}
\usepackage{stmaryrd}
\usepackage[pdfencoding=auto, psdextra, hidelinks]{hyperref}

\theoremstyle{plain}
\newtheorem{theorem}{Theorem}[section]
\newtheorem{lemma}[theorem]{Lemma}

\newtheorem{proposition}[theorem]{Proposition}
\newtheorem{hypothesis}[theorem]{Hypothesis}

\theoremstyle{definition}
\newtheorem{definition}[theorem]{Definition}

\theoremstyle{remark}
\newtheorem{remark}[theorem]{Remark}

\numberwithin{equation}{section}

\def\bold#1{\mbox{\boldmath $#1$}}
\newcommand{\uu}[1]{\bold{#1}}

\newcommand{\Left}{\left(}
\newcommand{\Right}{\right)}
\newcommand{\Langle}{\left\langle}
\newcommand{\Rangle}{\right\rangle}
\newcommand{\Lbrack}{\left\lbrack}
\newcommand{\Rbrack}{\right\rbrack}

\newcommand{\R}{\mathbb{R}}
\newcommand{\N}{\mathbb{N}}
\newcommand{\Cinf}{C_c^\infty}
\newcommand{\tdom}{\Left 0,T\Right}
\newcommand{\odom}{\tdom\times\Omega}
\newcommand{\dom}{\Lbrack 0,T\Rbrack\times\Omega}

\newcommand{\bu}{{\uu{u}}}
\newcommand{\bv}{\uu{v}}
\newcommand{\absu}{\abs{\uu{u}}}
\newcommand{\bm}{{\uu{m}}}
\newcommand{\absm}{\abs{\uu{m}}}
\newcommand{\bU}{{\uu{U}}}

\newcommand{\du}{\delta\bu}

\newcommand{\F}{\mathcal{F}} 
\newcommand{\D}{\mathcal{D}} 
\newcommand{\M}{\mathcal{M}} 
\newcommand{\K}{\mathcal{K}} 
\newcommand{\E}{\mathcal{E}}
\newcommand{\Pro}{\mathcal{P}} 
\newcommand{\veps}{\varepsilon}
\newcommand{\vrho}{\varrho}

\newcommand{\uphi}{\uu{\varphi}}

\newcommand{\half}{\frac{1}{2}}
\newcommand{\dx}{\mathrm{d}x}
\newcommand{\dt}{\mathrm{d}t}

\newcommand{\psig}{\psi_\gamma}
\newcommand{\V}{\mathcal{V}_{t,x}}
\newcommand{\vv}{\mathcal{V}}
\newcommand{\tvrho}{\tilde{\vrho}}
\newcommand{\tbm}{\widetilde{\bm}}
\newcommand{\X}{\mathcal{X}}
\newcommand{\m}{{(m)}}

\newcommand{\T}{\mathcal{T}}
\newcommand{\nuk}{\uu{\nu}_{\sigma,K}}
\newcommand{\Lt}{\mathcal{L}_{\mathcal{T}}}
\newcommand{\Ld}{\mathcal{L}_{\mathcal{D}}}
\newcommand{\flx}{F_{\sigma,K}}
\newcommand{\absk}{\abs{K}}
\newcommand{\abssig}{\abs{\sigma}}
\newcommand{\dsig}{D_{\sigma}}
\newcommand{\divup}{\mathrm{div}_{\T}^{up}}
\newcommand{\divt}{\mathrm{div}_{\T}}
\newcommand{\gradt}{\nabla_{\T}}
\newcommand{\gradd}{\nabla_{\D}}
\newcommand{\rk}{\vrho_K}
\newcommand{\uk}{\bu_K}
\newcommand{\pk}{p_K}
\newcommand{\sk}{\sigma,K}
\newcommand{\vsk}{v_{\sk}}
\newcommand{\sink}{\sigma\in\E(K)}
\newcommand{\delt}{\delta t}

\newcommand{\norm}[1]{\left\lVert#1\right\rVert}
\newcommand{\abs}[1]{\left\lvert#1\right\rvert}
\newcommand{\weakstar}{\overset{\ast}{\rightharpoonup}}

\newcommand{\Div}{\mathrm{div}_x}
\newcommand{\grad}{\nabla_x}
\newcommand{\Dt}{\partial_t}

\title[Semi-implicit FV Scheme for DMV Solutions to the Euler System]{A Semi-implicit Finite Volume Scheme for Dissipative Measure-valued Solutions to the Barotropic Euler System} 

\author[Arun and Krishnamurthy]{K. R. Arun and Amogh Krishnamurthy}

\address{School of Mathematics, Indian Institute of Science Education
  and Research Thiruvananthapuram, Thiruvananthapuram 695551, India} 
\email{arun@iisertvm.ac.in}

\address{School of Mathematics, Indian Institute of Science Education
  and Research Thiruvananthapuram, Thiruvananthapuram 695551, India} 
\email{amoghk0720@iisertvm.ac.in}

\date{\today}



\thanks{K. R. Arun acknowledges the support from Science and Engineering Research Board, Department of Science \& Technology, Government of India through grant CRG/2021/004078.}

\begin{document}

\begin{abstract}
    A semi-implicit in time, entropy stable finite volume scheme for the compressible barotropic Euler system is designed and analyzed and its weak convergence to a dissipative measure-valued (DMV) solution 
    [E.\ Feireisl et al., Dissipative measure-valued solutions to the compressible Navier-Stokes system, Calc.\ Var.\ Partial Differential Equations, 2016] of the Euler system is shown. The entropy stability is achieved by introducing a shifted velocity in the convective fluxes of the mass and momentum balances, provided some CFL-like condition is satisfied to ensure stability. A consistency analysis is performed in the spirit of the Lax's equivalence theorem under some physically reasonable boundedness assumptions. The concept of $\K$-convergence [E.\ Feireisl et al., $\K$-convergence as a new tool in numerical analysis, IMA J.\ Numer.\ Anal., 2020] is used in order to obtain some strong convergence results, which are then illustrated via rigorous numerical case studies. The convergence of the scheme to a DMV solution, a weak solution and a strong solution of the Euler system using the weak-strong uniqueness principle and relative entropy are presented.
\end{abstract}

\maketitle

\section{Introduction}
\label{sec:intro}

Hyperbolic systems of conservation laws arise in the modelling of various physical phenomena where finite speeds of propagation and conservation of quantities are involved. The most notable among them is, invariably, the non-linear Euler system of gas dynamics which represents the basic conservation laws of mass, momentum and energy. A common trait of non-linear conservation laws is that discontinuities in the form of shock waves may develop in finite time, even if the given initial data was sufficiently regular. As a result, the concept of weak or distributional solutions to conservation laws was introduced as an appropriate replacement. But, it is well known that these weak solutions may not be unique. As conservation laws usually model the real world, a selection criterion in the form of a relevant entropy condition is usually enforced in order to rule out the non-physical solutions. The enforced entropy condition is most often a mathematical equivalent of the second law of thermodynamics. This approach has yielded successful results in the scalar case; see \cite{BCP00, BL00, Kru70}. However, in light of recent results by Chiodaroli et al.\ \cite{CLK15} and De Lellis et al.\ \cite{LS10}, the uniqueness of entropy weak solutions fails to hold in the multidimensional case for the barotropic Euler system, and the result is extended to the complete Euler system by Feireisl et al.\ \cite{FKK+20}. The ill-posedness of the Euler system seems to stem from the lack of compactness of entropy weak solutions, in the sense that even a bounded sequence of solutions may develop oscillations or concentrations or both. 

Inspired by the previously stated results, there was a renewed interest in the concept of measure-valued (MV) solutions, introduced by DiPerna \cite{DiP85} in the context of general conservation laws. This was further extended to incompressible fluid flows by DiPerna and Majda \cite{DM87} and by Neustupa \cite{Neu93} for the compressible barotropic Euler and Navier-Stokes systems. The main advantage of MV solutions is that they can be identified as cluster points or limits of oscillatory approximated sequences; see the monograph \cite{MNR+96}. The measure in a MV solution refers to the Young measures being used to describe the possible oscillations in the approximate sequence. In the present work, we focus on the so-called dissipative measure-valued (DMV) solutions, first introduced by Feireisl et al.\ \cite{FGS+16} for the compressible Navier-Stokes system and later on for the Euler system in \cite{BF18, FGJ19}. The class of DMV solutions can be seen as a natural closure in the weak topology of the solution set to the underlying system, with the dissipative solutions being the barycenters of the associated Young measure. The major interest in this approach is because the phenomenon of `weak-strong uniqueness' seems to hold in the class of DMV solutions, i.e.\ a DMV solution and a strong solution emanating from the same initial data seemingly coincide provided the latter exists. We refer the interested reader to \cite{BF18, FGJ19, FGS+16, FLM+21a, GSW15, Wei18} for an exhaustive survey of the same.

Due to the existence of discontinuities in weak solutions, numerically approximating solutions of general conservation laws is challenging. Conventional finite difference approximations are known to be inappropriate near discontinuities where the PDE does not hold. The finite volume methods which captures discontinuities with high resolution present an attractive, robust and cost effective alternative to classical finite difference schemes. Over the past few decades a great deal of progress has been made about the finite volume methods and today, a variety of these methods are available. We refer the reader to \cite{EGH00, GR21, Lev90, Tor97} for an extensive survey of numerical methods for general hyperbolic problems. The convergence of finite volume methods to weak solutions of the compressible Euler system has been studied extensively in \cite{HKL14, HL19, HLN18, HLS21, HLC20}. The consistency analysis, in the spirit of Lax-Wendroff, is carried out across all the above mentioned references under strong assumptions, which may fail to hold in general. However, in the framework of DMV solutions, provided some consistency and stability estimates are established, one can show that the numerical solutions generated by the finite volume method converge, albeit in the weak sense, to a DMV solution of the Euler system under some physically reasonable boundedness assumptions. The current state of the art literature pertaining to the convergence of finite volume methods to DMV solutions and in general, MV solutions, includes \cite{FLM20b, FLM+19, FLM+21a, FLM+21b, FLM20a, FKM+17, FMT12, FMT16} and the references therein. 

Due to the convergence of the numerical solutions being weak and the limit being a DMV solution, there are difficulties in effectively capturing the limit solution. A celebrated result by Koml\'{o}s \cite{Kom67} asserts that any sequence of uniformly bounded $L^1$-functions $\lbrace f_j\rbrace_{j\in\N}$ admits a subsequence $\lbrace f_{j_k}\rbrace_{k\in\N}$ such that the arithmetical averages $\frac{1}{N}\sum_{k=1}^N f_{j_k}$ converge pointwise a.e.\ to a function $f$ in $L^1$, with any further subsequence of $\lbrace f_{j_k}\rbrace_{k\in\N}$ also enjoying the same property. Balder in \cite{Bal00} adapted this result to a measure-valued framework and introduced the novel concept of $\K$-convergence for sequences of Young measures. Using this, one can show that the arithmetical averages of the numerical solutions converge strongly to the expectations (or barycenters) of the Young measure associated to the DMV limit; see \cite{ FLM+21a, FLM+21b, FLM20a} for more details.

\subsection{Aims and Scope of the Paper}
\label{subsec:aim}

The aim of the present paper is to design and analyze a semi-implicit in time, entropy stable finite volume method for the compressible barotropic Euler system and show its weak convergence to a DMV solution. The main motivation in choosing a semi-implicit setup is that there is a considerable reduction in the computational complexity when implementing the scheme, as it avoids the need to invert large and dense matrices which is a common occurrence in fully-implicit schemes. Further, the restrictive stability conditions that are typical of fully-explicit schemes are also relaxed. The entropy stability is achieved by the introduction of a shifted velocity in the convective fluxes of the mass and momentum balances, with the shifted velocity being proportional to the pressure gradient, provided some CFL-like condition is satisfied.
A consistency analysis is performed in the spirit of the Lax's equivalence theorem under some boundedness assumptions, which reflect the numerical solutions staying in a non-degenerate region. The limiting process is then carried out in which we show that the sequence of numerical solutions contains a subsequence that generates a Young measure, and the subsequence further converges to the expectations (or barycenters) of said Young measure in the weak-* sense. We then employ the techniques of $\K$-convergence to obtain the convergence of the Ces\`{a}ro averages of the numerical solutions to the expectations of the Young measure in any $L^q$. We also deduce the convergence of the $s$-Wasserstein distance, $W_s$, between the Ces\`{a}ro averages of the Dirac masses centered at the numerical solutions and the Young measure in any $L^q$, along with the convergence of the absolute first variance or the $L^1$-deviations of the numerical solutions, which reflects the probabilistic nature of the limit solution.
Said convergences are then verified via rigorous numerical case studies, where we exhibit the convergence of the scheme to a DMV solution, a strong solution and a weak solution of the Euler system using the tools of the weak-strong uniqueness principle and relative entropy. The convergence to a weak solution is illustrated using a very specific setup wherein we consider a Riemann problem in one space dimension with a parameterized pressure law. Then, we simulate the pressure going to 0 by taking decreasing values of the above parameter. The vanishing pressure causes the formation of a $\delta$-shockwave, which is mathematically represented by a Dirac measure, leading to the solution being interpreted as a distributional solution of the system. 

\subsection{Barotropic Euler System}
\label{subsec:baro-sys}
We consider the following initial value problem for the barotropic Euler system in $\odom$ which reads
\begin{subequations}
\begin{align}
 &\Dt\vrho+\Div\Left\vrho\bu\Right\,=\,0, \label{eqn:mss-bal-baro} \\
 &\Dt\Left\vrho\bu\Right+\Div\Left\vrho\bu\otimes\bu\Right+\grad p\,=\,0, \label{eqn:mom-bal-baro} \\
 &\vrho(0,\cdot)=\vrho_0\,,\,\bu(0,\cdot)=\bu_0,
\end{align}
\end{subequations}
where the variables $\vrho = \vrho(t,x)$ and $\bu = \bu(t,x)$ are the density and velocity of the fluid respectively.
Here, $p=p(\vrho)$ denotes the pressure and for the sake of simplicity we consider the isentropic pressure law 
\begin{equation}
\label{eqn:baro-pres}
p(\rho)=a\vrho^\gamma,
\end{equation}
with $a>0$ and $\gamma>1$ denoting the ratio of specific heats also known as the adiabatic constant. We also set 
\begin{equation}
\label{eqn:pres-pot}
\psig(\vrho)=\frac{a}{\gamma-1}\vrho^\gamma
\end{equation}
as the so-called pressure potential or the internal energy per unit volume.

Note that $\psig$ along with $p$ satisfies 
\begin{equation}
\label{eqn:pres-pot-de}
z\psig^{\prime}(z)-\psig(z)=p(z).
\end{equation}
In the following, we consider the domain $\Omega\subset\R^d$, $d=1, 2$ or $3$, to be bounded together with the space-periodic boundary conditions.

We recall from \cite{HLS21} the following energy estimates satisfied by regular solutions of \eqref{eqn:mss-bal-baro}-\eqref{eqn:mom-bal-baro}
\begin{proposition}[A priori energy estimates]
\label{prop:baro-sol-idt}
Regular solutions  of \eqref{eqn:mss-bal-baro}-\eqref{eqn:mom-bal-baro} satisfy 
\begin{enumerate}
\item a renormalization identity:
\begin{equation}
\label{eqn:renorm-idt}
\Dt\psig(\vrho)+\Div(\psig(\vrho)\bu)+p\,\Div(\bu)\,=\,0.
\end{equation}

\item a kinetic energy identity:
\begin{equation}
\label{eqn:ke-idt}
\Dt\biggl(\half\vrho\absu^2\biggr)+\Div\biggl(\half\vrho\absu^2\bu\biggr)+\grad p\cdot\bu\,=\,0.
\end{equation}

\item the total energy identity/entropy identity:
\begin{equation}
\label{eqn:enrg-idt}
\Dt E+\Div((E+p)\bu)\,=\,0,
\end{equation}
where $E=\displaystyle\half\vrho\absu^2+\psig(\vrho)$.
\end{enumerate}
\end{proposition}

\subsection{Admissible Weak Solutions}
\label{subsec:wk-soln}

The weak formulation of \eqref{eqn:mss-bal-baro}-\eqref{eqn:mom-bal-baro} written in the conservative variables, namely the density $\vrho$ and the momentum $\bm=\vrho\bu$, reads 
\begin{equation}
\label{eqn:mss-wk-baro}
\Lbrack\int_\Omega\vrho\varphi\,\dx\Rbrack_{t=0}^{t=\tau}\,=\,\int_0^\tau\int_\Omega\Lbrack\vrho\Dt\varphi+\bm\cdot\grad\varphi\Rbrack\,\dx\,\dt,
\end{equation}
for any $\tau\in\Lbrack 0,T\Rbrack$ and $\varphi\in\Cinf\bigl(\lbrack 0,T)\times\Omega\bigr)$.
\begin{equation}
\label{eqn:mom-wk-baro}
\Lbrack\int_\Omega\bm\cdot\uu{\varphi}\,\dx\Rbrack_{t=0}^{t=\tau}\,=\,\int_0^\tau\int_\Omega\Lbrack\bm\cdot\Dt\uu{\varphi}+\Left\frac{\bm\otimes\bm}{\vrho}\Right\colon\grad\uu{\varphi}+p(\vrho)\,\Div\uu{\varphi}\Rbrack\,\dx\,\dt,
\end{equation}
for any $\tau\in\Lbrack 0,T\Rbrack$ and $\uphi\in\Cinf\bigl(\lbrack 0,T)\times\Omega;\R^d\bigr)$.

In addition, a weak solution is termed admissible if it satisfies the following total energy inequality:
\begin{equation}
\label{eqn:baro-entropy}
\frac{\mathrm{d}}{\dt}\int_\Omega\Lbrack\half\frac{\absm^2}{\vrho}+\psig(\vrho)\Rbrack\,\dx\leq 0.
\end{equation}

\subsection{Dissipative Measure-Valued Solutions} 
\label{subsec:dmv-baro}

Dissipative measure-valued (DMV) solutions can be seen as a generalization to the concept of weak solutions of the barotropic Euler system. We follow the formulation proposed in \cite{FLM+21a} in order to recall the notion of a DMV solution. The interested reader can refer to alternative formulations presented in \cite{BF18, FGJ19, FLM20b, FLM20a} and the references therein.

The phase space for the barotropic Euler system is defined as 
\begin{equation}
\label{eqn:phs-spc-baro}
\F=\bigl\{\lbrack \tvrho,\tbm\,\rbrack\colon\tvrho \geq 0,\tbm\in\R^d\bigr\} \subset{\R^{d+1}},
\end{equation}
and we let $\Pro(\F)$ denote the space of all proability measures on $\F$.

\begin{definition}[Dissipative measure-valued solution]
\label{defn:dmv-baro}
A parameterized family of probability measures $\mathcal{V}=\lbrace\V\rbrace_{(t,x)\in\odom}\in L_{weak-*}^\infty\bigl(\odom,\Pro(\F)\bigr)$ is called a DMV solution of the barotropic Euler system with initial data $\lbrack \vrho_0,\bm_0\rbrack$ if the following hold:

\begin{itemize}
\item \textbf{Energy inequality} - the integral inequality 
\begin{equation}
\label{eqn:enrg-ineq-baro}
\int_\Omega\biggl\langle \V ; \half\frac{\abs{\tbm}^2}{\tvrho}+\psig(\tvrho)\biggr\rangle\,\dx\,+\,\int_\Omega\,\mathrm{d}\mathfrak{C}_{cd}(t)\leq\int_\Omega\biggl\lbrack\half\frac{\abs{\bm_0}^2}{\vrho_0}+\psig(\vrho_0)\biggr\rbrack\,\dx
\end{equation}
holds for almost all $t\in\lbrack 0,T\rbrack$ with the so-called energy concentration defect measure
\begin{equation}
\label{eqn:enrg-con-def}
\mathfrak{C}_{cd}\in L^\infty\bigl(0,T;\M^+(\overline{\Omega})\bigr),
\end{equation}
where $\M^+(\overline{\Omega})$ denotes the set of all non-negative Radon measures on $\overline{\Omega}$;

\item \textbf{Equation of continuity} - the map
$\lbrack (t,x)\mapsto\bigl\langle\V ; \tvrho\bigr\rangle\rbrack\in C_{weak}(\lbrack 0,T\rbrack; L^\gamma(\Omega))$ and 
$\bigl\langle\mathcal{V}_{0,x};\tvrho\bigr\rangle = \vrho_0(x)$ for a.e. $x\in\Omega$. Further, the following integral identity 
\begin{equation}
\label{eqn:dmv-mss-bal-baro}
\biggl\lbrack\int_\Omega\bigl\langle\V;\tvrho\bigr\rangle\varphi\,\dx\biggr\rbrack_{t=0}^{t=\tau}\,=\,\int_0^\tau\int_\Omega\biggl\lbrack\bigl\langle\V;\tvrho\bigr\rangle\Dt\varphi+\bigl\langle\V;\tbm\bigr\rangle\cdot\grad\varphi\biggr\rbrack\,\dx\,\dt
\end{equation}
holds for any $\tau\in\lbrack 0,T\rbrack$ and $\varphi\in\Cinf(\lbrack 0,T)\times\Omega)$;

\item \textbf{Momentum equation} - the map 
$\lbrack (t,x)\mapsto\bigl\langle\V ; \tbm\bigr\rangle\rbrack\in C_{weak}(\lbrack 0,T\rbrack; L^\frac{2\gamma}{\gamma+1}(\Omega;\R^d))$ and  $\bigl\langle\mathcal{V}_{0,x};\tbm\bigr\rangle = \bm_0(x)\text{ for a.e. }x\in\Omega\,$. Further, the integral identity
\begin{equation}
\label{eqn:dmv-mom-bal-bro}
\begin{split}
&\Lbrack\int_\Omega\Langle\V,\tbm\Rangle\cdot\uu{\varphi}\,\dx\Rbrack_{t=0}^{t=\tau}\,\\
&=\,\int_0^\tau\int_\Omega\Lbrack\Langle\V,\tbm\Rangle\cdot\Dt\uu{\varphi}+\Langle\V,\frac{\tbm\otimes\tbm}{\tvrho}\Rangle\colon\grad\uu{\varphi}+\Langle\V,p(\tvrho)\Rangle\Div\uu{\varphi}\Rbrack\,\dx\,\dt \\
&+\int_0^\tau\int_\Omega\grad\uu{\varphi}\colon\mathrm{d}\mathfrak{R}_{cd}(t)\,\dt
\end{split}
\end{equation}
holds for any $\tau\in\lbrack 0,T\rbrack$ and $ \uphi\in\Cinf(\lbrack 0,T)\times\Omega;\R^d)$ with the so-called Reynolds concentration defect measure
\begin{equation}
\label{eqn:reyn-conc-def}
\mathfrak{R}_{cd}\in L^\infty\bigl(0,T;\M^{+}\bigl(\overline{\Omega};\R^{d\times d}_{sym}\bigr)\bigr);
\end{equation}

\item \textbf{Defect Compatibility Condition} - there exist constants $0<\underline{d}\leq\overline{d}$ such that
\begin{equation}
\label{eqn:enrg-reyn-dmv-rel}
\underline{d}\mathfrak{C}_{cd}\leq\mathrm{tr}(\mathfrak{R}_{cd})\leq\overline{d}\mathfrak{C}_{cd}.
\end{equation}
\end{itemize}
\end{definition}

\begin{remark}
Note that every weak solution $\lbrack\vrho,\bm\rbrack$ of the barotropic Euler system is also a DMV solution with concentration defect measures $\mathfrak{R}_{cd} = \mathfrak{C}_{cd} = 0$ and $\V = \delta_{\lbrack\vrho(t,x),\textbf{\textit{m}}(t,x)\rbrack}$ for a.e. $(t,x)\in\odom$.
\end{remark}

\subsection{Weak-Strong Uniqueness Principle}
\label{subsec:wk-strng-uni}

Though DMV solutions are much weaker objects when compared to weak solutions, the strong solutions of the Euler system are stable in the class of DMV solutions, i.e.\ any DMV solution and strong solution of the Euler system that emanate from the same initial data necessarily coincide, provided the latter exists. This fundamental property is labelled as the `weak-strong uniqueness principle' in literature; see  \cite{BF18, FGJ19, FGS+16, FLM+21a, GSW15, Wei18} for more details. The key tool in proving this principle is the so-called relative entropy functional, which is defined as 
\begin{equation}
\label{eqn:rel-ent}
E_{rel}(\vrho, \bm \,\vert\, r,\bU) = \half \vrho\abs{\frac{\bm}{\vrho} - \bU}^2 +\psig(\vrho) - \psig(r) - \psig^{\prime}(r)(\vrho - r),
\end{equation}
where $r,\bU$ are test functions mimicking a strong solution of the barotropic Euler system.

\begin{remark}
\label{rem:rel-ent}
Though \eqref{eqn:rel-ent} is not a metric as it is not symmetric, the following still hold:
\[
E_{rel}(\vrho, \bm \,\vert\, r,\bU)\geq 0\text{ and }E_{rel}(\vrho, \bm \,\vert\, r,\bU) = 0 \iff \vrho = r\,,\,\bm=r\bU.
\]
\end{remark}
We now recall the statement of the theorem following \cite{FLM+21a}.
\begin{theorem}[Weak-Strong Uniqueness Principle]
\label{thm:wk-str-uniq}
    Let a parameterized family of probability measures $\mathcal{V}=\bigl\{\V\bigr\}_{(t,x)\in\odom}$ be a DMV solution of the barotropic Euler system in the sense of Definition \ref{defn:dmv-baro}, with initial data $\lbrack\vrho_0,\bm_0\rbrack$ and  associated concentration defect measures $\mathfrak{C}_{cd}$ and $\mathfrak{R}_{cd}$.

    Let $r,\bU$ be a strong solution of the barotropic Euler system belonging to the class 
    \begin{equation}
    \label{eqn:baro-wk-strg-fun-cls}
        \begin{split}
            &r\in W^{1,\infty}(\odom),\,\inf_{(t,x)\in\odom}r(t,x)>0, \\
            &\bU\in W^{1,\infty}(\odom;\R^d),
        \end{split}
    \end{equation}
    such that $r(0,\cdot)=\vrho_0$ and $r(0,\cdot)\bU(0,\cdot)=\bm_0$.

    Then
    \[
        \V=\delta_{\lbrack r(t,x),r\textbf{\textit{U}}(t,x)\rbrack}\text{ for a.e. }(t,x)\in\odom,
    \]
    and
    \[
        \mathfrak{C}_{cd}=\mathfrak{R}_{cd}=0.
    \]
\end{theorem}

\subsection{Velocity Stabilization}
\label{subsec:vel-stab}

As mentioned earlier, the motivation for this work is to introduce a finite volume scheme to approximate the barotropic Euler system and to study its stability in the sense of energy stability, i.e.\ to obtain a discrete equivalent of \eqref{eqn:enrg-idt}. In order to enforce the stability at a numerical level, we adopt the formalism introduced in \cite{CDV17,DVB17,DVB20,GVV13,PV16} wherein a stabilization term in the form of a shifted velocity is introduced in the mass and momentum balances to yield the following modified system:
\begin{subequations}
\begin{align}
    &\Dt\vrho+\Div(\vrho(\bu-\du))=0 \label{eqn:mod-mss-bal-baro},\\
    &\Dt(\vrho\bu)+\Div(\vrho\bu\otimes(\bu-\du))+\grad p=0. \label{eqn:mod-mom-bal-baro}
\end{align}
\end{subequations}

Analogous to Proposition \ref{prop:baro-sol-idt}, we can show that the solutions to the modified system satisfy the following a priori energy estimates. 
\begin{proposition}[A priori energy estimates for the modified system]
    \label{prop:mod-enrg-est}
    Regular solutions of \eqref{eqn:mod-mss-bal-baro}-\eqref{eqn:mod-mom-bal-baro} satisfy
    \begin{enumerate}
        \item a renormalization identity:
        \begin{equation}
            \label{eqn:mod-renorm-idt}
            \Dt\psig(\vrho)+\Div(\psig(\vrho)(\bu-\du))+p\,\Div(\bu-\du)\,=\,0;
        \end{equation}

        \item a kinetic energy identity: 
        \begin{equation}
            \label{eqn:mod-ke-idt} \Dt\biggl(\half\vrho\absu^2\biggr)+\Div\biggl(\half\vrho\absu^2(\bu-\du)\biggr)+\grad p\cdot(\bu-\du)\,=\,-\grad p\cdot\du;
        \end{equation}
        
        \item the entropy balance:
        \begin{equation}
        \label{eqn:mod-ent-idt}
        \Dt E+\Div((E+p)(\bu-\du))\,=\,-\grad p\cdot\du.
        \end{equation}
    \end{enumerate}
\end{proposition}

Therefore, at least at the continuous level, we can see that formally setting $\du\,=\,\eta\grad p\,,\,\eta>0$, gives us
\begin{equation}
    \label{eqn:mod-ent-ineq}
    \Dt E+\Div((E+p)(\bu-\du))\,=\,-\eta\abs{\grad p}^2\leq 0.
\end{equation}
We thus observe that introducing a shift in the velocity indeed has a stabilizing effect, allowing us to obtain the desired entropy inequality.

\subsection{The Finite Volume Method}
\label{subsec:fv-mthd}

We wish to numerically approximate the velocity stabilized Euler system \eqref{eqn:mod-mss-bal-baro}-\eqref{eqn:mod-mom-bal-baro} and to this end, we start with a space domain $\Omega$ such that the closure of $\Omega$ is a union of closed rectangles ($d=2$) or closed parallelepipeds ($d=3$).

\subsubsection{Mesh and Unknowns}
\label{ssubsec:mesh-unkn}
We introduce a tessellation $\T$ of $\Omega\subset\R^d$, known as the primal mesh, consisting of possibly non-uniform closed rectangles ($d=2$) or closed parallelepipeds ($d=3)$ such that $\overline{\Omega}=\cup_{K\in\T}K$, where $K$ is called a control volume. The $d$-dimensional Lebesgue measure of $K$ will be denoted by $\absk$. The set of all edges ($d=2$) or faces ($d=3$) (hereafter commonly referred to as edges) of all control volumes $K\in\T$ is denoted by $\E$ and by $\abssig$, we denote the ($d-1$)- dimensional Lebesgue measure of $\sigma\in\E$. By $\E_{int}, \, \E_{ext}$ and $\E(K)$, we denote the subsets of all internal edges, external edges, i.e.\ edges lying on $\partial\Omega$, and the edges of a control volume $K\in\T$, respectively. Any two cells $K,L\in\T$ share a common edge denoted by $\sigma=K\vert L$. For $\sigma\in\E(K)$, $\nuk$ denotes the unit normal to $\sigma$ pointing outwards from $K$. By $a\lesssim b$ we mean $a\leq cb$ for a  constant $c>0$, independent of the mesh parameters.

The mesh size $h_\T$ is defined by 
\begin{equation}
\label{eqn:mesh-size}
    h_\T=\sup\lbrace h_K\colon K\in\T\rbrace,
\end{equation}
where $h_K = \text{diam}(K)$ and we assume that there exists a constant $\alpha>0$ (independent of other mesh parameters) such that 
\[
\alpha h_\T\leq\inf\lbrace h_K\colon K\in\T\rbrace. 
\]

By $\Lt(\Omega)\subset L^\infty(\Omega)$, we denote the space of all scalar-valued functions constant on each cell $K\in\T$ with the associated projection map $\Pi_\T\colon L^1(\Omega)\to\Lt(\Omega)$ defined as
\begin{equation}
    \label{eqn:proj-op-prim}
    \begin{split}
    &\Pi_\T q=\sum_{K\in\T}\bigl(\Pi_\T q\bigr)_K\X_K, \\
    &\bigl(\Pi_\T q\bigr)_K=\frac{1}{\absk}\int_K q\,\dx,
    \end{split}
\end{equation}
where $\X_K$ denotes the indicator function of $K$. Analogously, we define $\Lt(\Omega;\R^d)$, the space of all vector-valued piecewise constant functions with the projection operator being defined componentwise. 

For $q\in\Lt(\Omega)$, $q_K$ denotes the constant value of $q$ on the cell $K\in\T$. Further, if we have $\sigma = K\vert L\in\E(K)$, we define the average value of $q$ across $\sigma$ as 
\begin{equation}
    \label{eqn:avg-val-op}
    \overline{q}_{\sigma}=\frac{q_K+q_L}{2}.
\end{equation}

The discretization is collocated in the sense that the discrete unknowns are all associated to the same location. The discrete unknowns are thus $(\rk)_{K\in\T}$ for the density and $(\uk)_{K\in\T}$ for the velocity and we denote the pressure $\pk=p(\rk)=a\rk^\gamma$ for $K\in\T$.

\subsubsection{Discrete Mass Flux and Discrete Differential Operators}

\label{subsec:flx-dif-op}

We introduce the discrete mass flux and the discrete differential operators on the function space $\Lt(\Omega)$ defined above.

\begin{definition}[Discrete upwind mass flux]
    \label{def:mss-flx}
    For each $K\in\T$ and $\sigma\in\E(K)$, $\sigma=K\vert L$, the mass flux $\flx\colon\Lt(\Omega)\times\Lt(\Omega;\R^d)\to\R$ is defined as 
    \begin{equation}
        \label{eqn:mss-flx}
        \flx(q,\bv)=\abssig\bigl\lbrace q_K(\vsk)^+ + q_L(\vsk)^-\bigr\rbrace=\abssig\bigl\lbrace \flx^+ +\flx^-\bigr\rbrace,
    \end{equation}
    where we define 
    \[
    \vsk = 
    \begin{dcases}
        \overline{\bv}_\sigma\cdot\nuk, &\text{if }\sigma = K\vert L\in\E(K), \\
        0, &\text{if }\sigma\in\E(K)\cap\E_{ext}.
    \end{dcases}
    \]
    The positive and negative parts of $\vsk$ will be specified later once we give a description of the numerical scheme.
\end{definition}

\begin{definition}[Discrete gradient and divergence]
    \label{def:disc-grad-div}
    In the case of a collocated grid, we follow \cite{HLC20} in defining the discrete gradient and divergence operators so that the classical grad-div duality still holds at the discrete level.

    The discrete divergence operator $\divt\colon\Lt(\Omega;\R^d)\to\Lt(\Omega)$ is defined as 
    \begin{equation}
        \label{eqn:disc-div}
        \begin{split}
        &\divt \bv=\sum_{K\in\T}(\divt \bv)_K \X_K, \\
        &(\divt \bv)_K=\frac{1}{\absk}\sum_{\sink}\abssig\vsk,
        \end{split}
    \end{equation}

    while the discrete gradient operator $\gradt\colon\Lt(\Omega)\to\Lt(\Omega;\R^d)$ reads
    \begin{equation}
        \label{eqn:disc-grad}
        \begin{split}
        &\gradt q=\sum_{K\in\T}(\gradt q)_K\X_K, \\
        &(\gradt q)_K=\frac{1}{\absk}\Bigl\lbrack\,\sum_{\substack{\sink \\ \sigma=K\vert L}}\abssig\overline{q}_{\sigma}\,\nuk+\sum_{\sigma\in\E_{ext}\cap\E(K)}\abssig q_K\,\nuk\Bigr\rbrack.
        \end{split}
    \end{equation}
\end{definition}

We now report the following lemma from \cite{HLC20}. 
\begin{lemma}[Discrete grad-div duality]
    \label{lem:disc-grad-div-dual}
    The discrete gradient operator defined by \eqref{eqn:disc-grad} satisfies
    \begin{equation}
        \label{eqn:disc-grad2}
        (\gradt q)_K = \frac{1}{\absk}\sum_{\substack{\sink \\ \sigma = K\vert L}}\abssig\frac{q_L - q_K}{2}\nuk.
    \end{equation}
    Further, the discrete analogue of the classical grad-div duality holds, i.e.\ for any $q\in\Lt(\Omega)$ and $\bv\in\Lt(\Omega;\R^d)$ 
    \begin{equation}
        \label{eqn:disc-grad-div-dual}
        \int_\Omega\bigl(\gradt q\cdot\bv+q\,\divt\bv\bigr)\,\dx=0.
    \end{equation}
\end{lemma}

\begin{definition}[Upwind divergence operator]
    \label{def:up-div-op}
    We define an upwind divergence operator $\divup\colon\Lt(\Omega)\times\Lt(\Omega;\R^d)\to\Lt(\Omega)$ in order to discretize the convective terms in the mass and momentum balances which reads
    \begin{equation}
        \label{eqn:up-div-op}
        \begin{split}
            &\divup(q,\bv) = \sum_{K\in\T}(\divup(q,\bv))_K\X_K,\\
            &(\divup(q,\bv))_K=\frac{1}{\absk}\sum_{\substack{\sink \\ \sigma=K\vert L}}\flx(q,\bv).
        \end{split}
    \end{equation}
\end{definition}

\subsubsection{The Scheme} 
\label{ssubec-scheme}

Let us consider a discretization $0=t^0<t^1<\cdots<t^N=T$ of the time interval $\lbrack 0,T\rbrack$ and let $\delt=t^{n+1}-t^n$ for $n=0,1,\dots,N-1$ be the constant timestep. We initialize the scheme by taking the initial approximations to be the average values of our initial data $\vrho_0$ and $\bu_0$ respectively, i.e.\
\begin{equation}
    \label{eqn:disc-ini-dat}
    \begin{split}
        &\rk^0=(\Pi_\T\vrho_0)_K=\frac{1}{\absk}\int_K\vrho_0\,\dx, \\
        &\uk^0=(\Pi_\T\bu_0)_K=\frac{1}{\absk}\int_K\bu_0\,\dx.
    \end{split}
\end{equation}
Further, we also assume that $\vrho_0\in L^\infty(\Omega), \bu_0\in L^\infty(\Omega;\R^d)$ and $\vrho_0>0$ a.e.\ in $\Omega$.

We now consider the following fully discrete scheme for $0\leq n\leq N-1$:
\begin{subequations}
    \begin{align}
        &\frac{1}{\delt}(\rk^{n+1}-\rk^n)+(\divup(\vrho^{n+1},\bv^n))_K=0, \ \forall\,K\in\T, \label{eqn:disc-mss-bal}\\
        &\frac{1}{\delt}(\rk^{n+1}\uk^{n+1}-\rk^n\uk^n)+(\divup(\vrho^{n+1}\bu^n,\bv^n))_K+(\gradt p^{n+1})_K= 0, \ \forall\,K\in\T, \label{eqn:disc-mom-bal}
    \end{align}
\end{subequations}
where $\bv^n=\bu^n-\delta\bu^{n+1}$ denotes the stabilized velocity and we set $\delta\bu^{n+1}_K=\eta\delt(\gradt p^{n+1})_K$ with $\eta>0$, a parameter to be determined later; cf.\ \eqref{eqn:mod-ent-ineq}. The positive and negative parts of the term $\vsk$ appearing in the mass flux $\flx$ are defined so as to bring about an upwind-bias and maintain the signs split in $\flx$.
\begin{equation}
    \label{eqn:pos-neg-stab-velo}
    \begin{split}
    &(\vsk^n)^{+}=u_{\sk}^{n,+}-\delta u^{n+1,-}_{\sk}\geq 0, \\
    &(\vsk^n)^{-}=u_{\sk}^{n,-}-\delta u^{n+1,+}_{\sk}\leq 0,
    \end{split}
\end{equation}
wherein $a^{\pm}=\dfrac{a\pm\abs{a}}{2}$ denotes the standard positive and negative halves.

The discrete momentum update \eqref{eqn:disc-mom-bal} along with the discrete mass update \eqref{eqn:disc-mss-bal} yields the following update of the velocity component which is used in establishing the entropy stability.
\begin{equation}
    \label{eqn:disc-vel-upd}
    \frac{1}{\delt}(\uk^{n+1}-\uk^n)+\frac{1}{\absk}\sum_{\substack{\sink \\ \sigma=K\vert L}}\abssig\Biggl(\frac{\bu^n_L-\uk^n}{\rk^{n+1}}\Biggr) \flx^{n+1,-}+\frac{1}{\rk^{n+1}}(\gradt p^{n+1})_K=0, \ \forall\,K\in\T,
\end{equation}
where $\flx^{n+1,-}=\vrho_L^{n+1}(\vsk^n)^{-}$.

\subsection{Main Results and Organization}
\label{subsec-results}

The proposed scheme \eqref{eqn:disc-mss-bal}-\eqref{eqn:disc-mom-bal} involves solving a non-linear system at each time step, mainly due to the presence of non-linear stabilization terms in the mass balance. However, once we compute the updated density $\vrho^{n+1}$ from \eqref{eqn:disc-mss-bal}, the updated velocity $\bu^{n+1}$ can be computed in an explicit manner from the momentum balance \eqref{eqn:disc-mom-bal}. The existence of a discrete solution to the scheme is established following a standard topological degree argument, analogous to the treatment in \cite{AGK22,GMN19,Nir01,OCC06},  which uses classical tools from topological degree theory in finite dimensions; see \cite{Dei85}. For the sake of completeness, we just state the existence result.

\begin{theorem}[Existence of a Discrete Solution]
    \label{thm:exist-num-sol}
    Let $(\vrho^n,\bu^n)\in\Lt(\Omega)\times\Lt(\Omega;\R^{d})$ be such that $\vrho^n>0$ on $\Omega$. Then, there exists a solution $(\vrho^{n+1},\bu^{n+1})\in\Lt(\Omega)\times\Lt(\Omega;\R^{d})$ of \eqref{eqn:disc-mss-bal}-\eqref{eqn:disc-mom-bal} satisfying $\vrho^{n+1}>0$ on $\Omega$.
\end{theorem}

Once the existence of a discrete solution is known, we can define the notion of a numerical solution of the scheme as follows.

\begin{definition}[Numerical Solution of the Scheme]
    \label{def:scheme-soln}
    Let $\lbrace(\vrho^n,\bu^n)\in\Lt(\Omega)\times\Lt(\Omega;\R^d)\colon n=0,\dots,N-1\,;\,\vrho^n>0\rbrace$ be a family of discrete solutions to the scheme \eqref{eqn:disc-mss-bal}-\eqref{eqn:disc-mom-bal}. Then a numerical solution of the scheme corresponding to the space-time discretization $(\T,\delt)$ is defined to be the pair of piecewise constant functions 
    \begin{equation}
        \label{eqn:scheme-soln}
            \begin{split}
                &\vrho_{\T,\delt}(t,x)=\sum_{n=0}^{N-1}\sum_{K\in\T}\rk^n\X_K(x)\X_{[t_n,t_{n+1})}(t), \\
                &\bu_{\T,\delt}(t,x)=\sum_{n=0}^{N-1}\sum_{K\in\T}\uk^n\X_K(x)\X_{[t_n,t_{n+1})}(t).
            \end{split}
    \end{equation}
\end{definition}

Though Theorem \ref{thm:exist-num-sol} guarantees the positivity of density, we cannot claim or prove the existence of suitable a priori bounds that guarantee the numerical density being bounded below, away from zero or being bounded above. The density being bounded uniformly from above and below is crucial in the analysis carried out later. Hence, we impose the following as our principal working hypothesis; see also \cite{FLM20b,FLM+21a}. 

\begin{hypothesis}
\label{hyp:den-bound}
    There exist constants $\overline{\vrho}\geq\underline{\vrho}>0$ (independent of the mesh parameters) such that
    \begin{equation}
    \label{eqn:no-vac}
        0<\underline{\vrho}\leq\vrho_{\T,\delt}\leq\overline{\vrho}.
    \end{equation}
\end{hypothesis}

We can now state the main result concerning the entropy stability of the proposed scheme and we defer the proof to Section \ref{sec:ent-stab}.

\begin{theorem}[Local In-Time Entropy Inequality]
    \label{thm:disc-loc-ent-ineq}
1,    Any numerical solution to the scheme \eqref{eqn:disc-mss-bal}-\eqref{eqn:disc-mom-bal} satisfies the following entropy inequality
    \begin{equation}
        \label{eqn:disc-loc-ent-ineq}
        \sum_{K\in\T}\absk E^{n+1}_K\leq\sum_{K\in\T}\absk E^n_K,
    \end{equation}
    where $E^n_K = \displaystyle\half\rk^n\abs{\uk^n}^2 + \psig(\rk^n)$, for each $0\leq n\leq N-1$, under the conditions
    \begin{enumerate}
        \item $\eta\geq\dfrac{1}{\rk^{n+1}}$;
        
        \item $\delt\leq\dfrac{\absk\rk^{n+1}}{2\displaystyle\sum_{\substack{\sink \\ \sigma=K\vert L}}\abssig(-\flx^{n+1,-})}$.
    \end{enumerate}
\end{theorem}

In addition, we also get the following global entropy estimate. 

\begin{theorem}[Global Entropy Estimate]
    \label{thm:glob-ent-est}
    Assume that the hypothesis of Theorem \ref{thm:disc-loc-ent-ineq} holds. Then, there exists a constant $C>0$ such that the discrete solutions $(\vrho^n,\bu^n)_{0\leq n\leq N}$ satisfy the following global entropy inequality:
    \begin{equation}
    \label{eqn:disc-glob-ent-est}
        \sum_{K\in\T}\absk E^n_K + \sum_{r=0}^{n-1}\delt^2\sum_{K\in\T}\abs{K}\biggl(\eta - \frac{1}{\vrho^{r+1}_K}\biggr)\abs{(\gradt p^{r+1})_K}^2\leq C.
    \end{equation}
\end{theorem}

It is convenient to perform the consistency analysis of the scheme on the dual mesh. However, we remark that the dual mesh is not needed to implement the scheme. In order to perform the said analysis, in Section \ref{sec:cons}, we introduce the dual mesh $\D$ corresponding to the primal mesh $\T$. We also introduce a reconstruction operator $\mathcal{R}_\T$ which transforms a function which is piecewise constant on the control volumes of the primal mesh into a function which is piecewise constant on the dual cells; see \cite{HLC20} for more details. We now present a result related to the weak convergence of these dual mesh reconstructions and postpone the proof to Section \ref{sec:cons}.

\begin{lemma}[Weak Convergence of Dual Mesh Reconstructions]
    \label{lem:wk-cong-mesh-recon}
    Let $(\T^\m)_{m\in\N}$ be a sequence of space discretizations such that the mesh size $h_{\T^\m}\to 0$ as $m\to\infty$.
    Let $1\leq p<\infty$ and let $q\in L^p(\Omega)$. Let $\lbrace q^\m\rbrace_{m\in\N}$ be a uniformly bounded sequence in $L^p(\Omega)$ such that $q^\m\in\mathcal{L}_{\T^\m}(\Omega)$ for each $m\in\N$ and $q^\m\rightharpoonup q$ in $L^p(\Omega)$ as $m\to\infty$. 
    Let $\mathcal{R}_{\T^\m}$ denote the reconstruction operator on the mesh $\T^\m$. Then, $\mathcal{R}_{\T^\m} q^\m\rightharpoonup q$ in $L^p(\Omega)$ as $m\to\infty$.
\end{lemma}

Before we state the consistency result, we would like to remark that as a consequence of the timestep restriction imposed by Theorem \ref{thm:disc-loc-ent-ineq}, $\delt\to 0$ as $h_\T\to 0$. We now present the weak consistency result of our scheme and give a proof of the same in Section \ref{sec:cons}. 

\begin{theorem}[Consistency of the Numerical Scheme]
\label{thm:cons-scheme}
Let $(\T,\delt)$ be a given space-time discretization of our domain $\dom$. Let $(\vrho_{\T,\delt},\bu_{\T,\delt})$ denote a numerical solution to the scheme \eqref{eqn:disc-mss-bal}-\eqref{eqn:disc-mom-bal} in the sense of Definition \ref{def:scheme-soln} with the initial data given by \eqref{eqn:disc-ini-dat}. In addition to Hypothesis \ref{hyp:den-bound}, we make the following assumption:
\begin{itemize}    
    \item There exists a constant $\overline{u}>0$ (independent of the mesh parameters) such that 
    \begin{equation}
        \label{eqn:vel-cond}
        \abs{\bu_{\T,\delt}}\leq\overline{u}.
    \end{equation}
\end{itemize}

Then, the scheme \eqref{eqn:disc-mss-bal}-\eqref{eqn:disc-mom-bal} is consistent with the weak form of the Euler equations \eqref{eqn:mss-bal-baro}-\eqref{eqn:mom-bal-baro}, i.e.\
\begin{equation}
    \label{eqn:mss-cons}
    \begin{split}
        -\int_{\Omega}\vrho_{\T,\delt}(0,x)\varphi(0,x)\,\dx\,=\,\int_0^T\int_\Omega\bigl\lbrack &\vrho_{\T,\delt}\Dt\varphi
       +\vrho_{\D,\delt}\bu_{\D,\delt}\cdot\grad\varphi\bigr\rbrack\,\dx\,\dt + \mathcal{R}^{mass}_{\T,\delt}
    \end{split}
\end{equation}   
for any $\varphi\in\Cinf(\lbrack 0,T)\times \Omega)$;
\begin{equation}
    \label{eqn:mom-cons}
    \begin{split}
    -\int_\Omega\vrho_{\T,\delt}(0,x)\bu_{\T,\delt}(0,x)\cdot\uu{\varphi}(0,x)\,\dx\,&=\,\int_0^T\int_\Omega\bigl\lbrack\vrho_{\T,\delt}\bu_{\T,\delt}\cdot\grad\uu{\varphi} 
    +\vrho_{\D,\delt}(\bu_{\D,\delt}\otimes\bu_{\D,\delt})\colon\grad\uu{\varphi}\\
    &+p_{\T,\delt}(t+\delt, x)\Div\uu{\varphi}(t,x)\bigr)\bigr\rbrack\,\dx\,\dt + \mathcal{R}^{mom}_{\T,\delt}
    \end{split}
\end{equation}
for any $\uphi\in\Cinf(\lbrack 0, T)\times\Omega;\R^d)$.

Here, $\vrho_{\D,\delt}$ and $\bu_{\D,\delt}$ denote the reconstructions of $\vrho_{\T,\delt}$ and $\bu_{\T,\delt}$ on the dual mesh and the consistency errors $\mathcal{R}^{mass}_{\T,\delt}$ and $\mathcal{R}^{mom}_{\T,\delt}$ are such that

\begin{equation}
    \label{eqn:cons-errors}
    \begin{split}
        &\abs{\mathcal{R}^{mass}_{\T,\delt}} = o(h_\T,\sqrt{\delt}) \to 0\text{ as }h_\T\to 0,\\
        &\abs{\mathcal{R}^{mom}_{\T,\delt}} = o(h_\T,\sqrt{\delt}) \to 0\text{ as }h_\T\to 0.
    \end{split}
\end{equation}
\end{theorem}

We finally present the key result related to the weak convergence of the scheme and prove it in Section \ref{sec:conv}.

\begin{theorem}[Weak Convergence]
\label{thm:scheme-wk-con}
    Let $(\T^\m,\delt^\m)_{m\in\N}$ be a sequence of space-time discretizations such that $h^\m=h_{\T^\m}\to 0$ as $m\to\infty$. Let ($\vrho^\m,\bu^\m)_{m\in\N}$, $\vrho^\m = \vrho_{\T^\m,\delt^\m}$ and $\bu^\m = \bu_{\T^\m,\delt^\m}$, be the sequence of numerical solutions generated by the scheme. Suppose that the hypotheses of Theorem \ref{thm:disc-loc-ent-ineq} and Theorem \ref{thm:cons-scheme} are satisfied. Let $\bm^\m = \vrho^\m\bu^\m$. Then, there exists a subsequence (not relabelled) such that
    \begin{align*}
        &\vrho^\m\weakstar\langle\V; \tvrho\rangle\text{ in }L^\infty(\odom), \\
        &\bm^\m\weakstar\langle\V;\tbm\rangle\text{ in }L^\infty(\odom; \R^{d}),
    \end{align*}
    where $\vv=\lbrace\V\rbrace_{(t,x)\in\odom}$ is a DMV solution of the barotropic Euler system \eqref{eqn:mss-bal-baro}-\eqref{eqn:mom-bal-baro} (in the sense of Definition \ref{defn:dmv-baro}) with 
    \begin{equation}
        \mathfrak{C}_{cd} = 0\text{ and }\mathfrak{R}_{cd} = 0.
    \end{equation}
\end{theorem}

Once the weak convergence of the numerical solutions is achieved, we apply the techniques of $\K$-convergence in order to better visualize the limit solution. The key idea is that averaging the sequence of numerical solutions `compactifies' the sequence, which in turn allows us to obtain a strong limit (possibly for a subsequence). The following result can be obtained, and the proof follows analogous lines as in \cite{FLM+21a, FLM+21b}.

\begin{theorem}[$\K$-Convergence]
\label{thm:scheme-kcon}
    Let the assumptions of Theorem \ref{thm:scheme-wk-con} hold. Then, we have the following convergences (passing to another subsequence if needed).
    \begin{enumerate}[(i)]

        \item Strong convergence of the Ces\`{a}ro averages
        \[
        \begin{split}
        &\frac{1}{N}\sum_{m=1}^N\vrho^\m\to\langle\V;\tvrho\rangle\text{ in }L^q(\odom), \\
        &\frac{1}{N}\sum_{m=1}^N\bm^\m\to\langle\V;\tbm\rangle\text{ in }L^q(\odom;\R^d)
        \end{split}
        \]
        as $N\to\infty$, for any $1\leq q<\infty$;
        
        \item $L^q$-convergence of the s-Wasserstein distance  
        \[
        \norm{W_s\Biggl\lbrack\frac{1}{N}\sum_{m=1}^N\delta_{(\vrho^\m(t,x),\textbf{m}^\m(t,x))};\V\Biggr\rbrack}_{L^q(\odom)}\longrightarrow 0
        \]
        as $N\to\infty$, for any $1\leq q\leq s<\infty$;

        \item $L^1$-convergence of the deviations or the first variance
        \[
        \begin{split}
             &\frac{1}{N}\sum_{m=1}^N\norm{\vrho^\m-\frac{1}{N}\sum_{m=1}^N\vrho^\m}_{L^1(\odom)}\longrightarrow 0, \\
             &\frac{1}{N}\sum_{m=1}^N\norm{\bm^\m-\frac{1}{N}\sum_{m=1}^N\bm^\m}_{L^1(\odom;\R^d)}\longrightarrow 0
        \end{split}
        \]
        as $N\to\infty$.
    \end{enumerate}
\end{theorem}

Finally, in Section \ref{sec:num-res}, we present the results of the numerical case studies performed, and exhibit the convergences claimed in Theorem \ref{thm:scheme-kcon}.

\section{Entropy Stability of the Scheme}
\label{sec:ent-stab}

This entire section is devoted to the proof of Theorem \ref{thm:disc-loc-ent-ineq}, and we begin by proving the discrete analogue of Propositon \ref{prop:mod-enrg-est}. In what follows, by $\llbracket a,b\rrbracket$ we denote the interval $\lbrack\min\{a,b\},\max\{a,b\}\rbrack$.

\begin{theorem}
    \label{thm:disc-enrg-est}
    Any numerical solution to the scheme (in the sense of Definition \ref{def:scheme-soln}) satisfies the following discrete identities for each $0\leq n\leq N-1$:
    \begin{itemize}
        \item a discrete renormalization identity:
        \begin{equation}
            \label{eqn:disc-renorm}
            \frac{1}{\delt}(\psig(\rk^{n+1})-\psig(\rk^n))+(\divup(\psig(\vrho^{n+1}),\bv^n)_K+p^{n+1}_K(\divt\bv^n)_K+R^{n+1}_{K,\delt}=0,
        \end{equation}
        where the remainder term $R^{n+1}_{K,\delt}$ is given by
        \begin{equation}
            \label{eqn:renorm-rem}
            R^{n+1}_{K,\delt}=\frac{1}{2\delt}(\rk^{n+1}-\rk^n)^2\psig^{\prime\prime}(\overline{\vrho}_{K}^{n+1/2})+\frac{1}{2\absk}\sum_{\substack{\sink \\ \sigma=K\vert L}}\abssig(-(\vsk^n)^{-})(\rk^{n+1}-\vrho_L^{n+1})^2\psig^{\prime\prime}(\Tilde{\vrho}^{n+1}_{\sigma}),
        \end{equation}
        with $\overline{\vrho}_K^{n+1/2}\in\llbracket \rk^n,\rk^{n+1}\rrbracket$ and $\Tilde{\vrho}^{n+1}_{\sigma}\in\llbracket\rk^{n+1},\vrho_L^{n+1}\rrbracket$;\\

        \item a discrete kinetic energy identity:
        \begin{equation}
            \label{eqn:disc-kin}
            \begin{split}
            \frac{1}{\delt}\biggl(\half\rk^{n+1}\abs{\uk^{n+1}}^2-\half\rk^n\abs{\uk^n}^2\biggr)+\biggl(\divup\biggl(\half\vrho^{n+1}\abs{\bu^n}^2,\bv^n\biggr)\biggr)_K+&(\gradt p^{n+1})_K\cdot\bv^n_K+S^{n+1}_{K,\delt} \\
            &=-\eta\delt\abs{(\gradt p^{n+1})_K}^2,
            \end{split}
        \end{equation}
        where the remainder term $S^{n+1}_{K,\delt}$ is given by 
        \begin{equation}
            \label{eqn:ke-rem}
            S^{n+1}_{K,\delt}=-\frac{1}{2\delt}\rk^{n+1}\abs{\uk^{n+1}-\uk^{n}}^2+\frac{1}{2\absk}\sum_{\substack{\sink \\ \sigma=K\vert L}}\abssig\abs{\bu^n_L-\uk^n}^2(-\flx^{n+1,-});
        \end{equation}

        \item a discrete entropy identity:
        \begin{equation}
            \label{eqn:loc-ent-idt}
            \begin{split}
             \frac{1}{\delt}(E^{n+1}_K-E^n_K)+\biggl(\divup\biggl(\psig(\vrho^{n+1})+\half\vrho^{n+1}\abs{\bu^n}^2,\bv^n\biggr)\biggr)_K&+\Bigl(p^{n+1}_K(\divt\bv^n)_K+(\gradt p^{n+1})_K\cdot\bv^n_K\Bigr) \\
             &+(R^{n+1}_{K,\delt}+S^{n+1}_{K,\delt})=-\eta\delt\abs{(\gradt p^{n+1})_K}^2.
             \end{split}
        \end{equation}
    \end{itemize}
\end{theorem}

\begin{proof}
    See Appendix A1 for the proof of \eqref{eqn:disc-renorm} and Appendix A2 for the proof of \eqref{eqn:disc-kin}. Identity \eqref{eqn:loc-ent-idt} then follows as the sum of \eqref{eqn:disc-renorm} and \eqref{eqn:disc-kin}.
\end{proof}

We can now give the proof of Theorem \ref{thm:disc-loc-ent-ineq}.

\begin{proof}[Proof of Theorem \ref{thm:disc-loc-ent-ineq}]

Multiplying the discrete entropy identity \eqref{eqn:loc-ent-idt} by $\absk$ and then summing over all $K\in\T$, we obtain 
\begin{equation}
\label{eqn:sum-loc-ent}
    \sum_{K\in\T}\frac{\absk}{\delt}(E^{n+1}_K-E^n_K)+\sum_{K\in\T}\absk(R^{n+1}_{K,\delt}+S^{n+1}_{K,\delt})=-\sum_{K\in\T}\absk\eta\delt\abs{(\gradt p^{n+1})_K}^2
\end{equation}
wherein we have used the conservativity of the flux and the discrete duality \eqref{eqn:disc-grad-div-dual}.

Note that $\sum_{K\in\T}\absk R^{n+1}_{K,\delt}$ is unconditionally non-negative. The convexity of the map $y\mapsto\abs{y}^2$ allows us to apply the Jensen's inequality in the velocity update \eqref{eqn:disc-vel-upd} to yield
\begin{equation}
\label{eqn:est1}
    \frac{1}{2\delt}\rk^{n+1}\abs{\uk^{n+1}-\uk^n}^2\leq\frac{\delt}{\rk^{n+1}}\Biggl\lbrack\frac{1}{\absk^2}\Biggl(\sum_{\substack{\sink \\ \sigma=K\vert L}}\abssig(\bu^n_L-\uk^n)\flx^{n+1,-}\Biggr)^2+\abs{(\gradt p^{n+1})_K}^2\Biggr\rbrack.
\end{equation}
We can estimate the first term on the right hand side using the Cauchy-Schwarz inequality to get
\begin{equation}
\label{eqn:est2}
    \Biggl(\sum_{\substack{\sink \\ \sigma=K\vert L}}\abssig(\bu^n_L-\uk^n)\flx^{n+1,-}\Biggr)^2\leq \Biggl(\sum_{\substack{\sink \\ \sigma=K\vert L}}\abssig\abs{\bu^n_L-\uk^n}^2\flx^{n+1,-}\Biggr)\times\Biggl(\sum_{\substack{\sink \\ \sigma=K\vert L}}\abssig\flx^{n+1,-}\Biggr).
\end{equation}
Combining \eqref{eqn:ke-rem} along with the estimates \eqref{eqn:est1}-\eqref{eqn:est2} and using them in \eqref{eqn:sum-loc-ent}, we finally obtain

\begin{equation}
\label{eqn:est3}
    \begin{split}
        \sum_{K\in\T}\frac{\absk}{\delt}(E^{n+1}_K-E^n_K)\leq\sum_{K\in\T}\absk\Biggl(\sum_{\substack{\sink \\ \sigma=K\vert L}}\frac{\abssig}{\absk}&\abs{\bu^n_L-\uk^n}^2\flx^{n+1,-}\Biggr)\times\Biggl(\half +\frac{\delt}{\absk\rk^{n+1}}\sum_{\substack{\sink \\ \sigma=K\vert L}} \abssig\flx^{n+1,-}\Biggr) \\
        &+\delt\sum_{K\in\T}\absk\Biggl(\frac{1}{\rk^{n+1}}-\eta\Biggr)\abs{(\gradt p^{n+1})_K}^2.
    \end{split}
\end{equation}

The first term on the right hand side is non-positive under the timestep restriction
\begin{equation}
\label{eqn:cfl-cond}
    \delt\leq\dfrac{\absk\rk^{n+1}}{2\displaystyle\sum_{\substack{\sink \\ \sigma=K\vert L}}\abssig(-\flx^{n+1,-})},
\end{equation}
while the second term is non-positive under the condition $\eta\geq\frac{1}{\rk^{n+1}}$ which in turn yields the entropy inequality \eqref{eqn:disc-loc-ent-ineq}.
\end{proof}

The proof of Theorem \ref{thm:glob-ent-est} follows due to \eqref{eqn:est3}.

\begin{proof}[Proof of Theorem \ref{thm:glob-ent-est}]
    For each $r = 0,\dots, N - 1$, \eqref{eqn:est3} under the timestep restriction \eqref{eqn:cfl-cond} yields
    \[
    \sum_{K\in\T}\frac{\absk}{\delt}(E^{r+1}_K-E^r_K) + \delt\sum_{K\in\T}\absk\biggl(\eta - \frac{1}{\rk^{r+1}}\biggr)\abs{(\gradt p^{r+1})_K}^2 \leq 0.
    \]
    Multiplying by $\delt$ and summing over $r$ from $0$ to $n-1$, where $1\leq n\leq N$, we get
    \[
    \sum_{K\in\T}\absk E^n_K + \sum_{r = 0}^{n - 1}\delt^2\sum_{K\in\T}\absk\biggl(\eta - \frac{1}{\rk^{r+1}}\biggr)\abs{(\gradt p^{r+1})_K}^2\leq \sum_{K\in\T}\absk E^0_K\leq C,
    \]
    as $\vrho_0\in L^\infty(\Omega)$ and and $\bu_0\in L^\infty(\Omega;\R^d)$. 
\end{proof}

\begin{remark}
\label{rem:pres-grad-bound}
As a consequence of Theorem \ref{thm:glob-ent-est} and Hypothesis \ref{hyp:den-bound}, note that if we set $p_{\T,\delt} = p^n_K$ for $x\in K\in\T$ and $t\in\lbrack t_n,t_{n+1})$, $n=0,\dots, N-1$, we get
\begin{equation}
\label{eqn:pres-grad-bound}
\norm{\sqrt{\delt}\,\gradt p_{\T,\delt}}_{L^2}\leq C.
\end{equation}
\end{remark}

\section{Consistency Formulation of the Scheme}
\label{sec:cons}
The goal of this section is two-fold:\ one to establish the weak convergence of dual mesh reconstructions (Lemma \ref{lem:wk-cong-mesh-recon}) and the other to rigorously show that the proposed scheme is indeed consistent with the weak formulation of the barotropic Euler system (Theorem \ref{thm:cons-scheme}). In particular, we show that the numerical solutions satisfy the weak formulation of the continuous Euler system modulo some perturbation terms that vanish in the limit $h_\T\to 0$. As stated earlier, it is convenient to perform the consistency analysis on the dual mesh and for that purpose, we recall the notion of a dual mesh along with a discrete gradient defined therein; see \cite{GHL19, HLC20} for further details.

\subsection{Interpolates and Discrete Derivatives}
\label{subsec:interp-disc-der}
Let $\varphi\in\Cinf(\lbrack 0,T)\times\Omega)$. By $\varphi_{\T,\delt}$, we denote its interpolate for the space-time discretization ($\T,\delt$) defined by
\begin{equation}
    \label{eqn:spc-time-ipol}
    \varphi_{\T,\delt}(t,x)=\sum_{n=0}^{N-1}\varphi^n_K\X_{K}(x)\X_{[t_n,t_{n+1})}(t),
\end{equation}
where $\varphi^n_K=(\Pi_\T\varphi(t_n,x))_K=\frac{1}{\absk}\int_K\varphi(t_n,x)\,\dx$. If $\uu{\varphi}\in\Cinf(\lbrack 0,T)\times\Omega;\R^d)$, we define the interpolate componentwise.

We now define a discrete time derivative operator $\eth_t$
\begin{equation}
    \label{eqn:disc-time-der}
    \eth_t\varphi_{\T,\delt}(t,x)=\sum_{n=0}^{N-1}\sum_{K\in\T}\frac{\varphi^{n+1}_K-\varphi^n_K}{\delt}\X_K(x)\X_{[t_n,t_{n+1})}(t).
\end{equation}
The regularity of $\varphi$ ensures that the quantities $\varphi_{\T,\delt}, \eth_t\varphi_{\T,\delt}$ converge uniformly to $\varphi,\Dt\varphi$ respectively as $h_\T\to 0$; see also \cite{HLC20}.

\subsection{Dual Grid} 
\label{subsec:dual-grid}

 For each $\sigma\in\E_{int}$, $\sigma=K\vert L$, we associate a dual cell $D_\sigma=D_{K,\sigma}\cup D_{L,\sigma}$ where $D_{K,\sigma}$ (resp.\ $D_{L,\sigma})$ is built by half of $K$ (resp.\ $L$). If $\sigma\in\E_{ext}\cap\E(K)$, we define $D_\sigma=D_{K,\sigma}$. The collection of all dual cells is denoted by $\D=\lbrace D_\sigma\rbrace_{\sigma\in\E}$.

 Analogous to $\Lt(\Omega)$, the space of all scalar-valued (resp.\ vector-valued) functions constant on each dual cell $D_\sigma\in\D$ is denoted by $\Ld(\Omega)\subset L^\infty(\Omega)$ (resp.\ $\Ld(\Omega;\R^d)$).

\begin{definition}[Dual Mesh Gradient]
     \label{def:dual-grad}
     The discrete gradient operator on the dual grid, $\gradd\colon\Lt(\Omega)\to\Ld(\Omega;\R^d)$, is defined by 
     \begin{equation}
         \label{eqn:dual-grad}
         \begin{split}
             &\gradd q=\sum_{\substack{\sigma\in\E_{int} \\ \sigma=K\vert L}}(\gradd q)_{\sigma}\X_{\dsig}, \\
             &(\gradd q)_{\sigma}=\frac{\abssig}{\abs{\dsig}}(q_L-q_K)\nuk.
          \end{split}
        \end{equation}
\end{definition}

\subsection{Useful Estimates} 
\label{subsec:estimates}

A few basic inequalities used in the numerical analysis that follows are listed below and we refer the interested reader to \cite{EGH00, FLM+19, FLM+21a, GMN19} for more details. Assuming that $\varphi\in\Cinf(\Omega)$, $\uu{\varphi}\in\Cinf(\Omega;\R^d)$ and $1\leq p\leq\infty$, the following hold:
\begin{equation}
    \label{eqn:op-est}
    \begin{split}
    &\norm{\Pi_\T\varphi}_{L^p(\Omega)}\lesssim\norm{\varphi}_{\infty} , \ \norm{\varphi-\Pi_\T\varphi}_{L^p(\Omega)}\lesssim h_\T\norm{\varphi}_{\infty}, \\
    &\norm{\gradd\Pi_\T\varphi}_{L^p(\Omega;\R^d)}\lesssim\norm{\varphi}_{\infty} , \ \norm{\grad\varphi-\gradd\Pi_\T\varphi}_{L^p(\Omega;\R^d)}\lesssim h_\T ,\\
    &\norm{\Div\uu{\varphi}-\divt\Pi_\T\uu{\varphi}}_{L^p(\Omega)}\lesssim h_\T.
    \end{split}
\end{equation}

\subsection{Dual Mesh Reconstructions}
\label{subsec:dual-mesh-recon}

Following \cite{HLC20}, we introduce the notion of a reconstruction operator in order to mainly simplify the proof of the consistency result and recall a stability result of the same.

\begin{definition}[Reconstruction Operator]
    \label{def:recon-op}
    Given a primal mesh $\T$ of $\Omega$ and a corresponding dual mesh $\D$, a reconstruction operator is a map $\mathcal{R}_\T\colon\Lt(\Omega)\to\Ld(\Omega)$ defined as 
    \begin{equation}
        \label{eqn:recon-op}
        \mathcal{R}_\T q =\sum_{\sigma\in\E} \hat{q}_\sigma\X_{\dsig},
    \end{equation}
    where $\hat{q}_\sigma = \mu_\sigma q_K+(1-\mu_\sigma)q_L$ with $0\leq\mu_\sigma\leq 1$ for $\sigma=K\vert L\in\E_{int}$ and $\hat{q}_\sigma = q_K$ if $\sigma\in\E_{ext}\cap\E(K)$. 
\end{definition}

\begin{lemma}[Stability of the Reconstruction Operator]
    \label{lem:stab-recon-op}
    For any $1\leq p<\infty$, there exists $c\geq 0$ depending only on $p$ such that for any $q\in\Lt(\Omega)$,
    \begin{equation}
        \label{eqn:stab-recon-op}
        \norm{\mathcal{R}_\T q}_{L^p(\Omega)}\leq c\norm{q}_{L^p(\Omega)}.
    \end{equation}
\end{lemma}

We now establish the following estimate, which is analogous to the result \cite[Lemma 3.5]{GMN19}.
\begin{lemma}
    \label{lem:recon-grad-est}
    For any $1\leq p<\infty$, there exists $c>0$ depending only on $d$ and $p$ such that for any $q\in\Lt(\Omega)$,
    \begin{equation}
        \label{eqn:recon-grad-est}
        \norm{\mathcal{R}_\T q-q}_{L^p(\Omega)}\leq ch_\T\norm{\gradd q}_{L^p(\Omega)}.
    \end{equation}
\end{lemma}
\begin{proof}
    Noting that $\mathcal{R}_\T q - q = 0$ on $\dsig$ for each $\sigma\in\E_{ext}\cap \E(K)$, we get
    \begin{equation*}
        \begin{split}
            \norm{\mathcal{R}_\T q-q}_{L^p(\Omega)}^p &= \sum_{\substack{\sigma\in\E_{int} \\ \sigma = K\vert L}}\norm{\mathcal{R}_\T q-q}_{L^p(\dsig)}^p + \sum_{\sigma\in\E_{ext}}\norm{\mathcal{R}_\T q-q}_{L^p(D_{\sigma})}^p\\
            &=\sum_{\substack{\sigma\in\E_{int} \\ \sigma = K\vert L}}\Bigl(\norm{\mathcal{R}_\T q-q}_{L^p(D_{K,\sigma})}^p+\norm{\mathcal{R}_\T q-q}_{L^p(D_{L,\sigma})}^p\Bigr) \\
            &\leq\sum_{\substack{\sigma\in\E_{int} \\ \sigma = K\vert L}} \abs{\dsig}\abs{q_L-q_K}^p \\
            &\leq c h_{\T}^p\sum_{\substack{\sigma\in\E_{int} \\ \sigma = K\vert L}} \norm{\gradd q}_{L^p(D_\sigma)}^p \\
            &=c h_\T^p\norm{\gradd q}^p_{L^p(\Omega)},
        \end{split}
    \end{equation*}
    which establishes \eqref{eqn:recon-grad-est}.
\end{proof}

We can now give the proof of Lemma \ref{lem:wk-cong-mesh-recon}.

\begin{proof}[Proof of Lemma \ref{lem:wk-cong-mesh-recon}]

For the sake of convenience, let $\mathcal{R}_{\T^\m}=\mathcal{R}_m$, $\Pi_{\T^\m}=\Pi_m$ and $h_{\T^\m} = h^\m$. It is enough to show that $\lVert{\mathcal{R}_m q^\m - q^\m}\rVert_{L^p(\Omega)}\to 0$ as $m\to\infty$ to prove the claim.

Using the density of $\Cinf(\Omega)$ in $L^p(\Omega)$ and the fact that $\lbrace q^\m\rbrace_{m\in\N}$ is a uniformly $L^p$-bounded sequence, given $\veps>0$, we obtain $\varphi^\m\in \Cinf(\Omega)$ such that $\lVert\varphi^\m - q^\m\rVert_{L^p(\Omega)}<\veps$ for each $m\in\N$ and that the sequence $\lbrace\varphi^\m\rbrace_{m\in\N}$ is uniformly bounded. Also, since $h^\m\to 0$ as $m\to\infty$, there exists $M\in\N$ such that $h^\m<\veps$ for all $m\geq M$. 

Then, noting that $\mathcal{R}_m\Pi_m q^\m =\mathcal{R}_m q^\m$ as $q^\m\in\mathcal{L}_{\T^\m}(\Omega)$ for each $m\in\N$, we get
\begin{equation}
\label{eqn:seq-est}
    \begin{split}
        \lVert\mathcal{R}_m q^\m-q^\m\rVert_{L^p(\Omega)} &\leq \lVert\mathcal{R}_m\Pi_m q^\m-\mathcal{R}_m\Pi_m\varphi^\m\rVert_{L^p(\Omega)} + \lVert\mathcal{R}_m\Pi_m\varphi^\m - \Pi_m\varphi^\m\rVert_{L^p(\Omega)} \\
        &+ \lVert\Pi_m\varphi^\m-\varphi^\m\rVert_{L^p(\Omega)} + \lVert\varphi^\m-q^\m\rVert_{L^p(\Omega)}.
    \end{split}
\end{equation}
Every term on the right hand side can be estimated using a combination of Lemma \ref{lem:stab-recon-op}, Lemma \ref{lem:recon-grad-est} and the estimates provided by \eqref{eqn:op-est} as follows: 
\begin{align*}
&\lVert\mathcal{R}_m\Pi_m q^\m-\mathcal{R}_m\Pi_m\varphi^\m\rVert_{L^p(\Omega)}\leq c \lVert\varphi^\m-q^\m\rVert_{L^p(\Omega)}<c\veps, \\ 
&\lVert\mathcal{R}_m\Pi_m\varphi^\m - \Pi_m\varphi^\m\rVert_{L^p(\Omega)}\leq c h^\m\lVert\nabla_{\D^\m}\Pi_m\varphi^\m\rVert_{L^p(\Omega)}< c \veps, \ \forall\,m\geq M, \\
&\lVert\Pi_m\varphi^\m-\varphi^\m\rVert_{L^p(\Omega)}\leq c h^\m< c\veps, \ \forall\,m\geq M.
\end{align*}
Combining all the above estimates, we get
\[
\lVert\mathcal{R}_m q^\m-q^\m\rVert_{L^p(\Omega)}<c\veps, \ \forall\,m\geq M.
\]
Hence, $\lVert\mathcal{R}_m q^\m-q^\m\rVert_{L^p(\Omega)}\to 0$ as $m\to\infty$ and this proves the claim.
\end{proof}

We now have all the necessary tools required to prove Theorem \ref{thm:cons-scheme}.

\begin{proof}[Proof of Theorem \ref{thm:cons-scheme}]

First, we prove the consistency of the mass balance \eqref{eqn:disc-mss-bal}. Let $\varphi\in\Cinf(\lbrack 0,T)\times\Omega)$ and let $\varphi_{\T,\delt}$ denote its interpolate on the primal grid given by \eqref{eqn:spc-time-ipol}. Multiply the mass balance \eqref{eqn:disc-mss-bal} with $\absk\delt\varphi^{n}_K$ and sum over all $K\in\T$ and $0\leq n\leq N-1$ to obtain 
\[
T_1+T_2=0,
\]
where 
\begin{equation}
\label{eqn:mss-cons-split}
    \begin{split}
        & T_1 = \sum_{n=0}^{N-1}\sum_{K\in\T}\absk(\rk^{n+1}-\rk^n)\varphi^{n}_K, \\
        & T_2 = \sum_{n=0}^{N-1}\delt\sum_{K\in\T}\sum_{\substack{\sink \\ \sigma=K\vert L}}\abssig\Bigl(\flx^{n+1,+}+\flx^{n+1,-}\Bigr)\varphi^{n}_K.
    \end{split}
\end{equation}
Analogous to \cite[Theorem 3.5]{HLN18} and \cite[Theorem 4.5]{HLS21}, we can write $T_1$ as 
\begin{equation}
\label{eqn:T1}
    T_1 = -\int_0^T\int_\Omega\vrho_{\T,\delt}\eth_t\varphi_{\T,\delt}\,\dx\,\dt - \int_\Omega\vrho_{T,\delt}(0,x)\varphi_{\T,\delt}(0,x)\,\dx.
\end{equation}
$T_2$ is again split as $T_2 = T_{2,1} + T_{2,2}$, where 
\begin{equation}
\label{eqn:T2-split}
    \begin{split}
        &T_{2,1} = \sum_{n=0}^{N-1}\delt\sum_{K\in\T}\sum_{\substack{\sink \\ \sigma=K\vert L}}\abssig\Bigl(\rk^{n+1} u^{n,+}_{\sk}+\vrho^{n+1}_L u^{n,-}_{\sk}\Bigr)\varphi^{n}_K, \\
        &T_{2,2} = -\sum_{n=0}^{N-1}\delt\sum_{K\in\T}\sum_{\substack{\sink \\ \sigma=K\vert L}}\abssig\Bigl(\rk^{n+1} \delta u^{n+1,-}_{\sk}+\vrho^{n+1}_L\delta u^{n+1,+}_{\sk}\Bigr)\varphi^{n}_K.
    \end{split}
\end{equation}
We can reorder the summation in $T_{2,1}$ to get
\begin{equation}
    \label{eqn:T21-reord}
    T_{2,1} = -\sum_{n=0}^{N-1}\delt\sum_{\substack{\sigma\in\E_{int} \\ \sigma=K\vert L}}\abs{\dsig}\vrho^{n+1}_\sigma\overline{\bu}^n_\sigma\cdot\frac{\abssig}{\abs{\dsig}}(\varphi^{n}_L-\varphi^{n}_K)\nuk.
\end{equation}
We now set $\vrho_{\D,\delt}$ and $\bu_{\D,\delt}$ as piecewise constant functions on the dual cells that are equal to $\vrho^{n+1}_\sigma$ (the upwind choice) and $\overline{\bu}^n_\sigma$ respectively on each dual cell $\dsig$ and time interval $\lbrack t_n,t_{n+1})$, which are the reconstructions of $\vrho_{\T,\delt}$ and $\bu_{\T,\delt}$ on the dual grid. Then, we can rewrite \eqref{eqn:T21-reord} as
\begin{equation}
\label{eqn:T21-final}
    T_{2,1} = -\int_0^T\int_\Omega\vrho_{\D,\delt}\bu_{\D,\delt}\cdot\gradd\varphi_{\T,\delt}\,\dx\,\dt.
\end{equation}
Combining the terms, we obtain 
\begin{equation}
\label{eqn:mss-cons2}
    \begin{split}
        -\int_{\Omega}\vrho_{\T,\delt}(0,x)\varphi(0,x)\,\dx\,=\,\int_0^T\int_\Omega\bigl\lbrack &\vrho_{\T,\delt}\Dt\varphi
        +\vrho_{\D,\delt}\bu_{\D,\delt}\cdot\grad\varphi\bigr\rbrack\,\dx\,\dt + \mathcal{R}^{mass}_{\T,\delt},
    \end{split}
\end{equation}
where 
\begin{equation}
\label{eqn:mss-cons-err}
    \begin{split}
        &\mathcal{R}^{mass}_{\T,\delt} = T_{2,2}+\,\int_0^T\int_\Omega \vrho_{\T,\delt}(\eth_t\varphi_{\T,\delt}-\Dt\varphi)\,\dx\,\dt 
        +\int_{0}^T\int_\Omega\vrho_{\D,\delt}\bu_{\D,\delt}\cdot(\gradd\varphi_{T,\delt}-\grad\varphi)\,\dx\,\dt \\
        &+\int_{\Omega}\vrho_{\T,\delt}(0,x)(\varphi_{\T,\delt}(0,x)-\varphi(0,x))\,\dx\  \\
        & = T_{2,2}+R_1+R_2+R_3.
    \end{split}
\end{equation}
The terms $R_1,R_3$ are readily estimated using the strong convergence of the interpolates $\varphi_{\T,\delt}, \eth\varphi_{\T,\delt}$ to $\varphi,\Dt\varphi$ respectively along with the presumed assumptions. To estimate $R_2$, we use the estimate provided by \eqref{eqn:op-est} to obtain 
\[
\abs{R_2}\leq \overline{\vrho}\,\overline{u}\int_0^T\int_\Omega\abs{\gradd\varphi_{T,\delt}-\grad\varphi}\,\dx\,\dt\lesssim h_\T.
\]
Finally, $T_{2,2}$ can be simplified in the same manner as $T_{2,1}$ and further, using Lemma \ref{lem:stab-recon-op} and the pressure gradient estimate \eqref{eqn:pres-grad-bound} along with Hypothesis \ref{hyp:den-bound}, we get

\[
\abs{T_{2,2}}\leq \overline{\vrho}\eta\sqrt{\delt}\norm{\varphi}_{\infty}\norm{\sqrt{\delt}\gradt p_{\T,\delt}}_{L^2}\lesssim\sqrt{\delt}.
\]

Thus, we obtain that $\displaystyle\abs{\mathcal{R}^{mass}_{\T,\delt}}=o(h_\T,\sqrt{\delt})$. 

Next, we show the consistency of the momentum balance \eqref{eqn:disc-mom-bal}. Let $\uphi\in\Cinf(\lbrack 0,T)\times\Omega;\R^d)$ and let $\uphi_{\T,\delt}$ denote its interpolate on the primal grid given by \eqref{eqn:spc-time-ipol}. Take the dot product of the momentum balance \eqref{eqn:disc-mom-bal} with $\absk\delt\uphi^n_K$ and sum over all $K\in\T$ and $0\leq n\leq N-1$ to obtain 
\[
S_1+S_2+S_3 = 0,
\]
where 
\begin{equation}
 \label{eqn:mom-bal-split}
    \begin{split}
        &S_1 = \sum_{n=0}^{N-1}\sum_{K\in\T}\absk(\vrho^{n+1}_K\bu^{n+1}_{K}-\rk^n\uk^n)\cdot\uphi^n_K, \\
        &S_2 = \sum_{n=0}^{N-1}\delt\sum_{K\in\T}\sum_{\substack{\sink \\ \sigma = K\vert L}}\abssig\Bigl(F^{n+1,+}_{\sk}\uk^n+F^{n+1,-}_{\sk}\bu^n_L\Bigr)\cdot\uphi^n_K, \\
        &S_3 = \sum_{n=0}^{N-1}\delt\sum_{K\in\T}\absk(\gradt p^{n+1})_K\cdot\uphi^n_K.
    \end{split}
\end{equation}
Again, proceeding as in \cite[Theorem 3.5]{HLN18} and \cite[Theorem 4.5]{HLS21}, we can write $S_1$ as 
\begin{equation}
    \label{eqn:S1}
        S_1 = -\int_0^T\int_\Omega\vrho_{\T,\delt}\bu_{\T,\delt}\cdot\eth_t\uphi_{\T,\delt}\,\dx\,\dt - \int_\Omega\vrho_{T,\delt}(0,x)\bu_{\T,\delt}(0,x)\cdot\uphi_{\T,\delt}(0,x)\,\dx.
\end{equation}
Now, we use the discrete grad-div duality \eqref{eqn:disc-grad-div-dual} to rewrite $S_3$ as 
\begin{equation}
\label{eqn:S3}
    S_3 = -\int_0^T\int_\Omega p_{\T,\delt}(t+\delt,x)\divt\uphi_{\T,\delt}(t,x)\,\dx\,\dt.
\end{equation}

Analogous to the mass balance, we split $S_2$ as $S_2 = S_{2,1}+S_{2,2}$ with the simplified form of $S_{2,1}$ given by  
\begin{equation}
    \label{eqn:S2-split}
    S_{2,1} = -\int_0^T\int_\Omega\vrho_{\D,\delt}(\bu_{\D,\delt}\otimes\bu_{\D,\delt})\colon\gradd\uphi_{\T,\delt}\,\dx\,\dt,
\end{equation}
and $S_{2,2}$ reads
\begin{equation}
    S_{2,2} = -\sum_{n=0}^{N-1}\delt\sum_{K\in\T}\sum_{\substack{\sink \\ \sigma=K\vert L}}\abssig\Bigl(\rk^{n+1}\uk^n \delta u^{n+1,-}_{\sk}+\vrho^{n+1}_L\bu^n_L\delta u^{n+1,+}_{\sk}\Bigr)\cdot\uphi^{n}_K.
\end{equation}
Combining the terms, we obtain 
\begin{equation}
\label{eqn:mom-cons2}
    \begin{split}
    -\int_\Omega\vrho_{\T,\delt}(0,x)\bu_{\T,\delt}(0,x)\cdot\uu{\varphi}(0,x)\,\dx\,&=\,\int_0^T\int_\Omega\bigl\lbrack\vrho_{\T,\delt}\bu_{\T,\delt}\cdot\grad\uu{\varphi} 
    +\vrho_{\D,\delt}(\bu_{\D,\delt}\otimes\bu_{\D,\delt})\colon\grad\uu{\varphi}\\
    &+p_{\T,\delt}(t+\delt, x)\Div\uu{\varphi}(t,x)\bigr)\bigr\rbrack\,\dx\,\dt + \mathcal{R}^{mom}_{\T,\delt},
    \end{split}
\end{equation}
where
\begin{equation}
\label{mom-cons-err}
    \begin{split}
        &\mathcal{R}^{mom}_{\T,\delt} = S_{2,2}+\,\int_0^T\int_\Omega \vrho_{\T,\delt}\bu_{\T,\delt}\cdot(\eth_t\uphi_{\T,\delt}-\Dt\uphi)\,\dx\,\dt \\
        &+ \int_{0}^T\int_\Omega\vrho_{\D,\delt}(\bu_{\D,\delt}\otimes\bu_{\D,\delt})\colon(\gradd\uphi_{T,\delt}-\grad\uphi)\,\dx\,\dt \\
        &+\int_0^T\int_\Omega p_{\T,\delt}(t+\delt,x)(\divt\uphi_{\T,\delt}(t,x)-\Div\uphi(t,x))\,\dx\,\dt \\
        &+ \int_{\Omega}\vrho_{\T,\delt}(0,x)\bu_{\T,\delt}(0,x)\cdot(\uphi_{\T,\delt}(0,x)-\uphi(0,x))\,\dx
        = S_{2,2}+Q_1+Q_2+Q_3+Q_4.
    \end{split}
\end{equation}
$Q_1, Q_4$ are readily estimated due to the strong convergence of the interpolants to the respective classical variants along with the presumed assumptions. $Q_2$ and $Q_3$ are handled using the estimates provided by \eqref{eqn:op-est} as follows: 
\begin{equation*}
    \begin{split}
        &\abs{Q_2}\leq\overline{\vrho}\,\overline{u}^2\int_{0}^T\int_\Omega\abs{\gradd\uphi_{\T,\delt}-\grad\uphi}\,\dx\,\dt\lesssim h_\T, \\
        &\abs{Q_3}\leq a\overline{\vrho}^\gamma\int_0^T\int_\Omega\abs{\divt\uphi_{\T,\delt}-\Div\uphi}\,\dx\,\dt\lesssim h_\T.
    \end{split}
\end{equation*}
Finally, $S_{2,2}$ can be estimated using Lemma \ref{lem:stab-recon-op} and the pressure gradient estimate \eqref{eqn:pres-grad-bound} along with Hypothesis \ref{hyp:den-bound} to yield

\[
\abs{S_{2,2}}\leq\overline{\vrho}\,\overline{u}\eta\sqrt{\delt}\norm{\uphi}_{\infty}\norm{\sqrt{\delt}\gradt p_{\T,\delt}}_{L^2}\lesssim\sqrt{\delt}.
\]
Thus, we obtain $\displaystyle\abs{\mathcal{R}^{mom}_{\T,\delt}}=o(h_\T, \sqrt{\delt})$ and this completes the proof.
\end{proof}

\section{Convergence Of The Scheme}
\label{sec:conv}
This entire section is dedicated to the proof of Theorem \ref{thm:scheme-wk-con}. For a fixed space-time discretization $(\T,\delt)$, we set $E_{\T,\delt}(t,x) = E^n_K$ for $x\in K\in\T$ and $t\in\lbrack t_n,t_{n+1})$, $n=0,\dots, N-1$.

\begin{proof}[Proof of Theorem \ref{thm:scheme-wk-con}]

As a consequence of the assumptions made in Hypothesis \ref{hyp:den-bound} and Theorem \ref{thm:cons-scheme}, denoting $p_{\T^\m,\delt^\m} = p^\m$ and $E_{\T^\m,\delt^\m} = E^\m$, we have the estimates
\begin{equation}
\label{eqn:seq-clas}
    \begin{split}
        &(\vrho^\m)_{m\in\N},\, (p^\m)_{m\in\N}, (E^\m)_{m\in\N}\subset L^\infty(\odom), \\
        &(\bm^\m)_{m\in\N}\subset L^\infty(\odom;\R^d), \\
        &\Biggl(\frac{\bm^\m\otimes\bm^\m}{\vrho^\m}\Biggr)_{m\in\N}\subset\ L^\infty(\odom;\R^{d\times d}).
     \end{split}
\end{equation}
Applying the fundamental theorem of Young measures, cf.\ \cite{JBal89, Ped97}, yields the existence of a subsequence (not relabelled) that generates a Young measure $\vv=\lbrace\V\rbrace_{(t,x)\in\odom}$ such that 
\begin{align*}
    &\vrho^\m\weakstar\langle\V; \tvrho\rangle\text{ in }L^\infty(\odom), \\
    &\bm^\m\weakstar\langle\V;\tbm\rangle\text{ in }L^\infty(\odom; \R^{d}).
\end{align*}
Moreover, as a consequence of the estimates \eqref{eqn:seq-clas}, we get
\begin{equation*}
    \begin{split}
        &\frac{\bm^\m\otimes\bm^\m}{\vrho^\m}\weakstar \Bigl\langle\V; \frac{\tbm\otimes\tbm}{\tvrho}\Bigr\rangle\text{ in }L^\infty(\odom;\R^{d\times d}), \\
        &p^\m\weakstar \langle\V;\tilde{p}\rangle\text{ in }L^\infty(\odom), \\
        &E^\m\weakstar \langle\V;\tilde{E}\rangle\text{ in }L^\infty(\odom),
     \end{split}
\end{equation*}
where $\tilde{p} = p(\tvrho)$ and $\tilde{E} = E(\tvrho,\tbm)$. Hence, we get the required weak precompactness of these sequences and as a consequence, we have $\mathfrak{C}_{cd} = \mathfrak{R}_{cd} = 0$.

Passing to the limit $m\to \infty$ along with the use of Lemma \ref{lem:wk-cong-mesh-recon} in \eqref{eqn:mss-cons} and \eqref{eqn:mom-cons} yields 
\begin{equation}
\label{eqn:mss-baro-dmv}
\biggl\lbrack\int_\Omega\bigl\langle\V;\tvrho\bigr\rangle\varphi\,\dx\biggr\rbrack_{t=0}^{t=\tau}\,=\,\int_0^\tau\int_\Omega\biggl\lbrack\bigl\langle\V;\tvrho\bigr\rangle\Dt\varphi+\bigl\langle\V;\tbm\bigr\rangle\cdot\grad\varphi\biggr\rbrack\,\dx\,\dt,
\end{equation}
for any $\tau\in\lbrack 0,T\rbrack$ and for any $\varphi\in\Cinf(\lbrack 0,T)\times\Omega)$,
\begin{equation}
    \label{eqn:mom-baro-dmv}
    \begin{split}
        &\Lbrack\int_\Omega\Langle\V,\tbm\Rangle\cdot\uu{\varphi}\,\dx\Rbrack_{t=0}^{t=\tau}\,\\
        &=\,\int_0^\tau\int_\Omega\Lbrack\Langle\V,\tbm\Rangle\cdot\Dt\uu{\varphi}+\Langle\V,\frac{\tbm\otimes\tbm}{\tvrho}\Rangle\colon\grad\uu{\varphi}+\Langle\V,p(\tvrho)\Rangle\Div\uu{\varphi}\Rbrack\,\dx\,\dt, \\
    \end{split}
\end{equation}
for any $\tau\in\lbrack 0,T\rbrack$ and $\uphi\in\Cinf(\lbrack 0,T)\times\Omega;\R^d)$.

Further, as a direct consequence of Theorem \ref{thm:disc-loc-ent-ineq}, we obtain that for a.e $\tau\in\lbrack 0,T\rbrack$,
\[
\int_\Omega E^\m (\tau,\cdot)\,\dx \leq\int_{\Omega}E^\m(0,\cdot)\,\dx.
\]
Passing to the limit and noting that $\mathfrak{C}_{cd} = 0$, 
\begin{equation}
\label{eqn:enrg-baro-dmv}
    \int_\Omega\langle\vv_{\tau,x};\tilde{E}\rangle\,\dx\,\dt\leq \int_\Omega E_0\,\dx,
\end{equation}
holds for almost all $\tau\in\lbrack 0,T\rbrack$ where $E_0 =\half\frac{\abs{\textbf{\textit{m}}_0}^2}{\vrho_0} + \psig(\vrho_0)$, proving that $\vv$ is indeed a DMV solution of the Euler system which concludes the proof.
\end{proof}

\section{Numerical Results}
\label{sec:num-res}

The goal of this section is to present the results of the numerical case studies. The implementation of the scheme \eqref{eqn:disc-mss-bal}-\eqref{eqn:disc-mom-bal} is done as follows. The mass balance \eqref{eqn:disc-mss-bal} is solved first in order to obtain the updated density $\vrho^{n+1}$, which is done using an iterative Newton method due to the presence of nonlinear terms. Then, the momentum equation \eqref{eqn:disc-mom-bal} is evaluated explicitly to get the updated velocity $\bu^{n+1}$.

In accordance with Theorem \ref{thm:disc-loc-ent-ineq}, note that the timestep $\delt$ needs to satisfy 
\begin{equation}
\label{eqn:time-step-comp}
\frac{\delt}{\absk}\sum_{\substack{\sink \\ \sigma =  K\vert L}}\abssig\frac{(-\flx^{n+1,-})}{\vrho^{n+1}_K}\leq\half,
\end{equation}
and $\eta\geq(\rk^{n+1})^{-1}$. Both these conditions are implicit in nature and since the former involves the flux as well, they are difficult to implement in practice. However, we remark that if the following implicit restriction holds
\[
\frac{\delt}{\absk}\sum_{\substack{\sink \\ \sigma =  K\vert L}}\abssig\frac{\abs{\flx(\vrho^{n+1}, \bv^n)}}{\vrho^{n+1}_K}\leq\half,
\]
then,
\[
\frac{3}{2}\rk^{n+1}-\rk^n \geq \rk^{n+1}-\rk^n+\frac{\delt}{\absk}\sum_{\substack{\sink \\ \sigma = K\vert L}}\abssig\abs{\flx(\vrho^{n+1}, \bv^n)}\geq 0.
\]
Therefore, we get that 
\[
\frac{3}{2\rk^n} \geq \frac{1}{\rk^{n+1}}
\]
and as a consequence, choosing 
\begin{equation}
\label{eqn:eta-suff-cond}
    \eta\geq\dfrac{3}{2\rk^n}
\end{equation} 
satisfies the necessary condition $\eta \geq (\rk^{n+1})^{-1}$.

We now state a sufficient timestep restriction which is easy to implement in practice for the scheme \eqref{eqn:disc-mss-bal}-\eqref{eqn:disc-mom-bal}.

\begin{proposition}
\label{prop:suff-time-step-comp}
    For each $\sigma = K\vert L\in\E(K)$, suppose the following holds:
    \begin{equation}
    \label{eqn:suff-time-step-comp}
        \delt\max\biggl\lbrace\frac{\abs{\partial K}}{\abs{K}},\frac{\abs{\partial L}}{\abs{L}}\biggr\rbrace\biggl\lbrace\abs{\overline{\bu}^n_\sigma} +\sqrt{\eta\abs{\frac{(\gradt p^{n+1})_K + (\gradt p^{n+1})_L}{2}}}\biggr\rbrace\leq\min\biggl\lbrace 1, \frac{1}{3}\frac{\min\lbrace\rk^n, \vrho^n_L\rbrace}{\max\lbrace\vrho^{n+1}_K,\vrho^{n+1}_L\rbrace} \biggr\rbrace,
    \end{equation}
    where $\abs{\partial K} = \sum_{\sink}\abssig$. Then, $\delt$ satisfies \eqref{eqn:time-step-comp}.
\end{proposition}

\begin{proof}
    The proof follows exactly as given in \cite[Proposition 3.2]{CDV17} and hence we omit the details.
\end{proof}

\begin{remark}
    Although the sufficient timestep condition \eqref{eqn:suff-time-step-comp} is still implicit in nature, we implement it explicitly during the computations as also done in \cite{AGK22, DVB17}.
\end{remark}

In all the numerical experiments that follow, we consider a uniform Cartesian mesh. Let $U_k$ denote the solution computed on a mesh of $k$ points for a $1D$ problem and a grid of $k\times k$ points for a $2D$ problem. Let $\overline{U}_k$ and $\Tilde{U}_k$ respectively denote the Ces\`{a}ro average of the numerical solutions and their first variance, i.e. 
\[
\overline{U}_k = \frac{1}{k}\sum_{j=1}^k U_j,\,\,\,\,\,\Tilde{U}_k = \frac{1}{k}\sum_{j=1}^k\abs{U_j-\overline{U}_k}.
\]

Let $U_K$ be the reference solution computed on the finest possible mesh. In the experiments that follow, the reference solution is always computed using an explicit Rusanov scheme. Analogous to \cite{FLM+21a, FLM+21b}, we compute the four errors 
\[
E_1 = \|U_k-U_{K}\|,\, E_2 = \|\overline{U}_k-\overline{U}_{K}\|,\, E_3 = \|\Tilde{U}_k-\Tilde{U}_{K}\|,\, E_4 = \|W_1(\overline{\vv}^k_{t,x},\overline{\vv}^{K}_{t,x})\|.
\]
where $\norm{\cdot}$ denotes the $L^1$-norm and $\overline{\vv}^k_{t,x} = \frac{1}{k}\sum_{j=1}^k\delta_{U_j(t,x)}$, where $\delta_{U_j(t,x)}$ denotes the Dirac measure centered at $U_j(t,x)$ for $(t,x)\in\odom$. Further, we also compute the errors
\[
E_5 = \|\overline{U}_{k}-\overline{U}_{K}\|_2,\, E_6 = \|W_1(\overline{\vv}^{k}_{t,x},\overline{\vv}^{K}_{t,x})\|_2
\]
in order to corroborate the claims made in Theorem \ref{thm:scheme-kcon}, where $\|\cdot\|_2$ denotes the $L^2$-norm. According to Theorem \ref{thm:scheme-kcon}, we only have $L^1$-convergence of the 1-Wasserstein distance. However, we also observe $L^2$-convergence in the experiments that follow and hence, report the same.

\begin{remark}
    $W_1$ above denotes the 1-Wasserstein distance and it is defined as follows:\ Given probability measures $\mu$ and $\nu$ on a subset $Q\subset\R^d$, the 1-Wasserstein distance between them is defined as in \cite{FLM+21a}.
    \[
    W_1(\mu,\nu) = \inf_{\lambda\in\Pro(Q\times Q), P_1\lambda = \mu, P_2\lambda = \nu}\int_{Q\times Q}\abs{x - y}\mathrm{d}\lambda(x,y),
    \]
    where $\Pro(Q\times Q)$ denotes the set of probability measures on $Q\times Q$ and $P_1\lambda, P_2\lambda$ represent the projections (marginals) of the measure $\lambda\in\Pro(Q\times Q)$, i.e.\ $P_1\lambda(S) = \lambda(S\times Q)$ and $P_2\lambda(S) = \lambda(Q\times S)$ for any subset $S$ of $Q$. In the experiments that follow, this distance is computed using the {\texttt{wasserstein\textunderscore distance}} function available in the {\texttt{scipy.stats}} package of the open source library SciPy, cf.\ \cite{JOP01}.
\end{remark}

\subsection{Cylindrical Explosion} 
\label{subsec:cyl-exp}

We consider the $2D$ cylindrical explosion problem from \cite{DLV17} with the pressure law given as $p(\vrho) = \vrho^{1.4}$. The initial density reads 
\begin{equation*}
    \vrho(0,x_1,x_2) = 
    \begin{dcases}
        2, &r^2\leq\frac{1}{4}, \\
        1, &\text{otherwise},
    \end{dcases}
\end{equation*}
where $r^2 = x_1^2+x_2^2$. The initial velocity field is taken as 
\begin{equation*}
    (u_1,u_2)(0,x_1,x_2) = -\frac{\alpha(x_1,x_2)}{\vrho(0,x_1,x_2)}\biggl(\frac{x_1}{r},\frac{x_2}{r}\biggr)\X_{r > 10^{-15}},
\end{equation*}
where $\alpha(x_1,x_2) = \max\{0, 1-r\}(1-e^{-16r^2})$. We consider the computational domain $\lbrack -1,1\rbrack\times\lbrack -1,1\rbrack$ with periodic boundaries everywhere. The  final time is $T=0.25$ and $K = 2048$ in this case. 

In Figure \ref{fig:cyl-exp-plot}, we present the pseudocolor plots of the density $\vrho$ and the momentum $\bm = (m_1,m_2)$, with $k = 1024$, which clearly shows the resolution of the circularly expanding shock wave. In Figure \ref{fig:cyl-exp-err}, we present the plots of the errors $E_1,\dots, E_6$ for the density and the momentum. We observe that classical convergence of individual numerical solutions is not obtained, in particular due to the increase in $E_1$ error for the momentum components. However, the approach using $\K$-convergence yields the convergence of the Ces\`{a}ro averages along with the convergence of the $L^1$-deviations and the Wasserstein distance, leading us to conclude that the numerical solutions indeed converge to a DMV solution of the Euler system. Also, the claims made regarding the convergence of the Ces\`{a}ro averages and the Wasserstein distance in any $L^q$-norm seem to hold, described by the errors $E_5$ and $E_6$ respectively. The values of the errors $E_1,\dots,E_6$ of $\vrho$ along with their convergence rates are presented in Table \ref{tab:ord-den-cyl-exp}. As an additional observation, one can see that the errors of the momentum components end up coinciding due to the symmetric nature of the problem. The maximum value of $\eta$ observed for this problem was $1.5$.

\begin{figure}[htpb]
    \centering
    \includegraphics[height = 0.18\textheight]{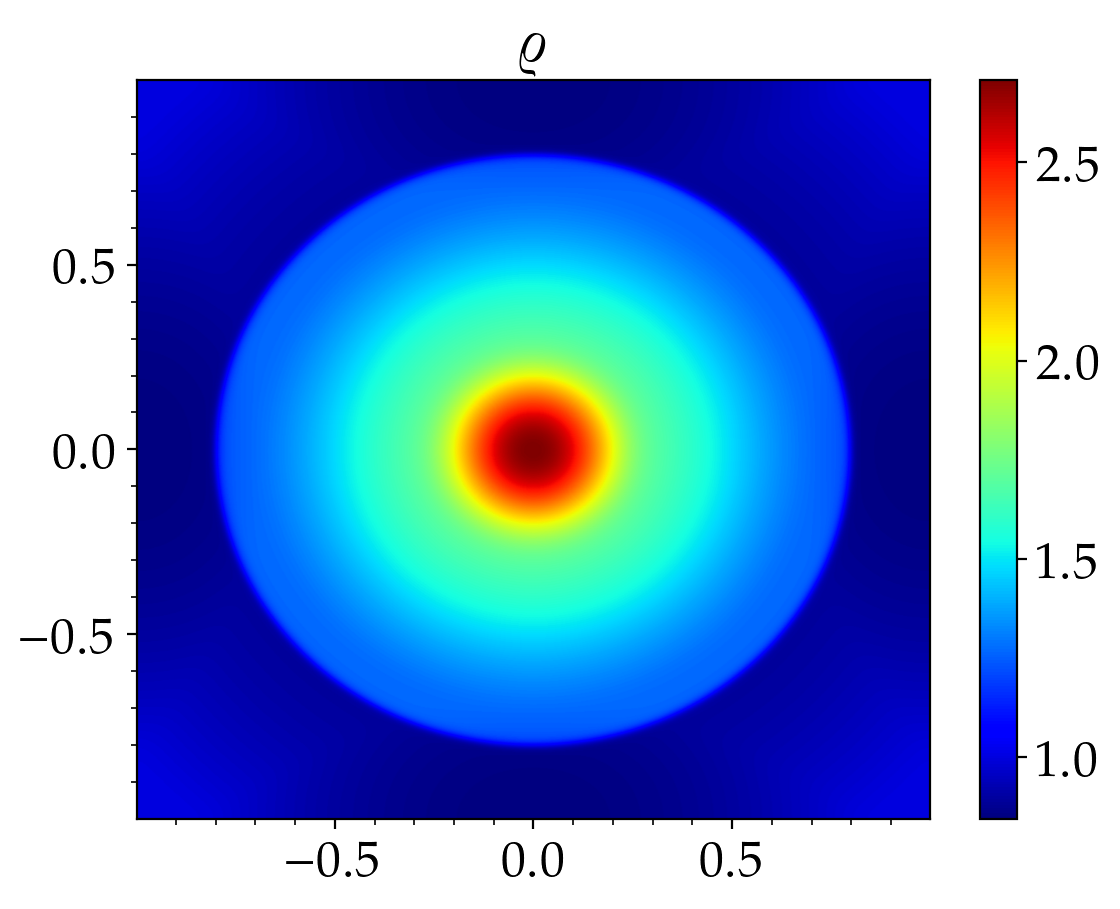}
    \includegraphics[height = 0.18\textheight]{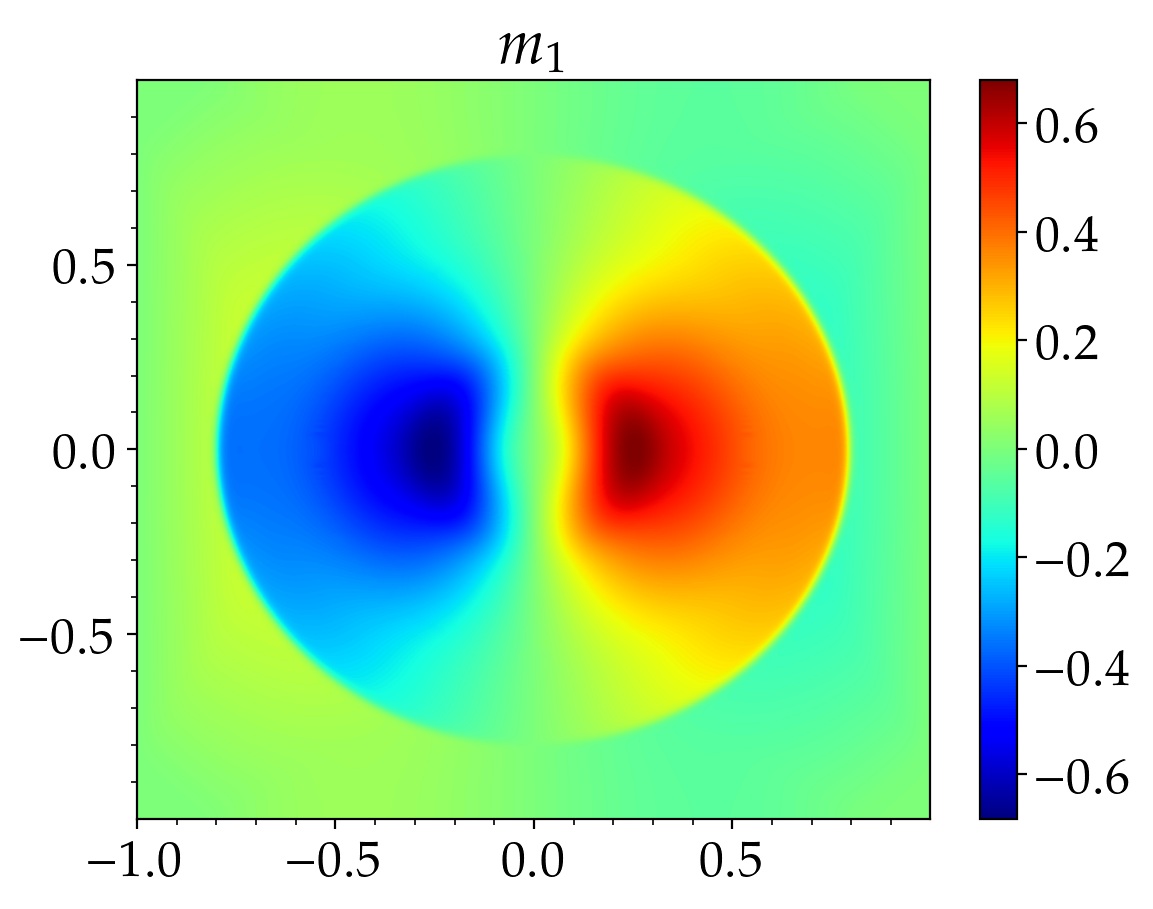}
    \includegraphics[height = 0.18\textheight]{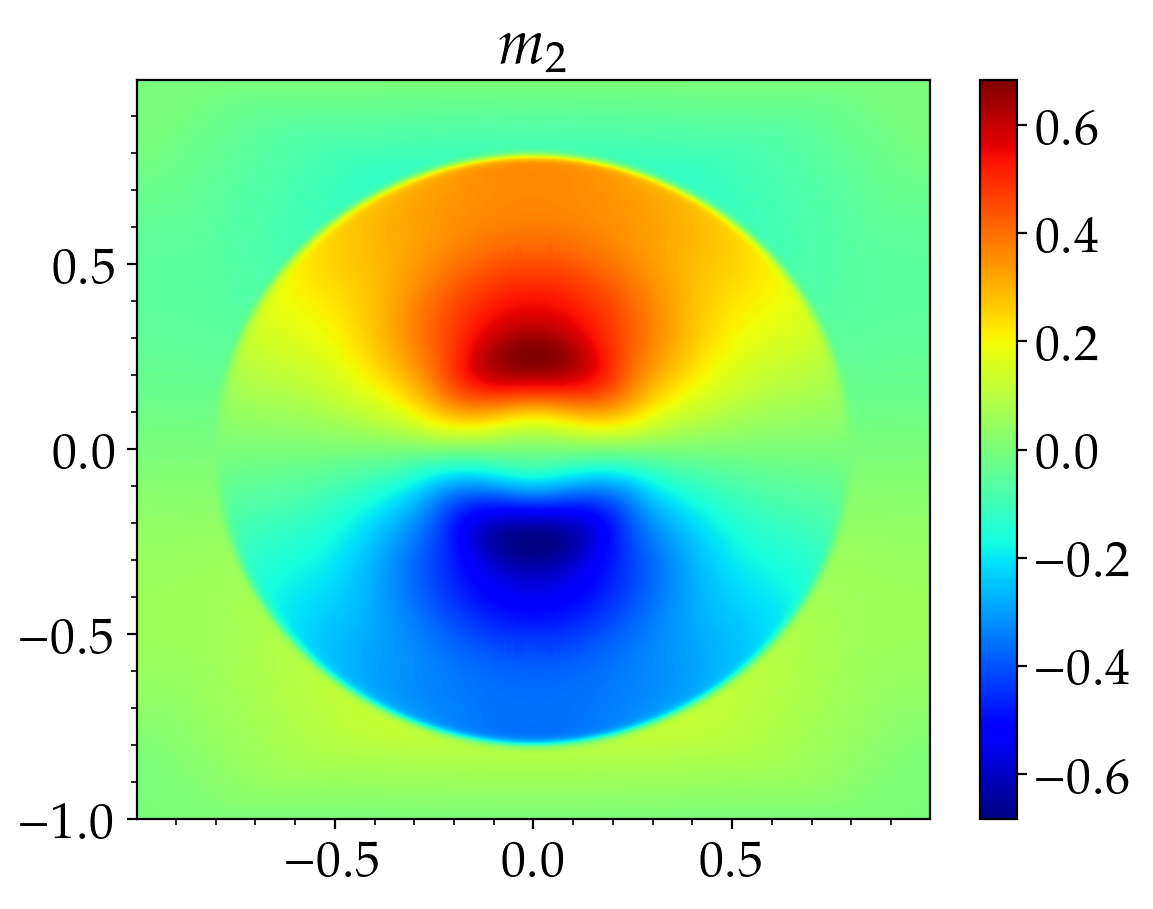}
    \caption{Density and momentum profiles of the cylindrical explosion problem with $k = 1024$ at time $T=0.25$.}
    \label{fig:cyl-exp-plot}
\end{figure}

\begin{figure}[htpb]
    \centering
    \includegraphics[height = 0.4\textheight]{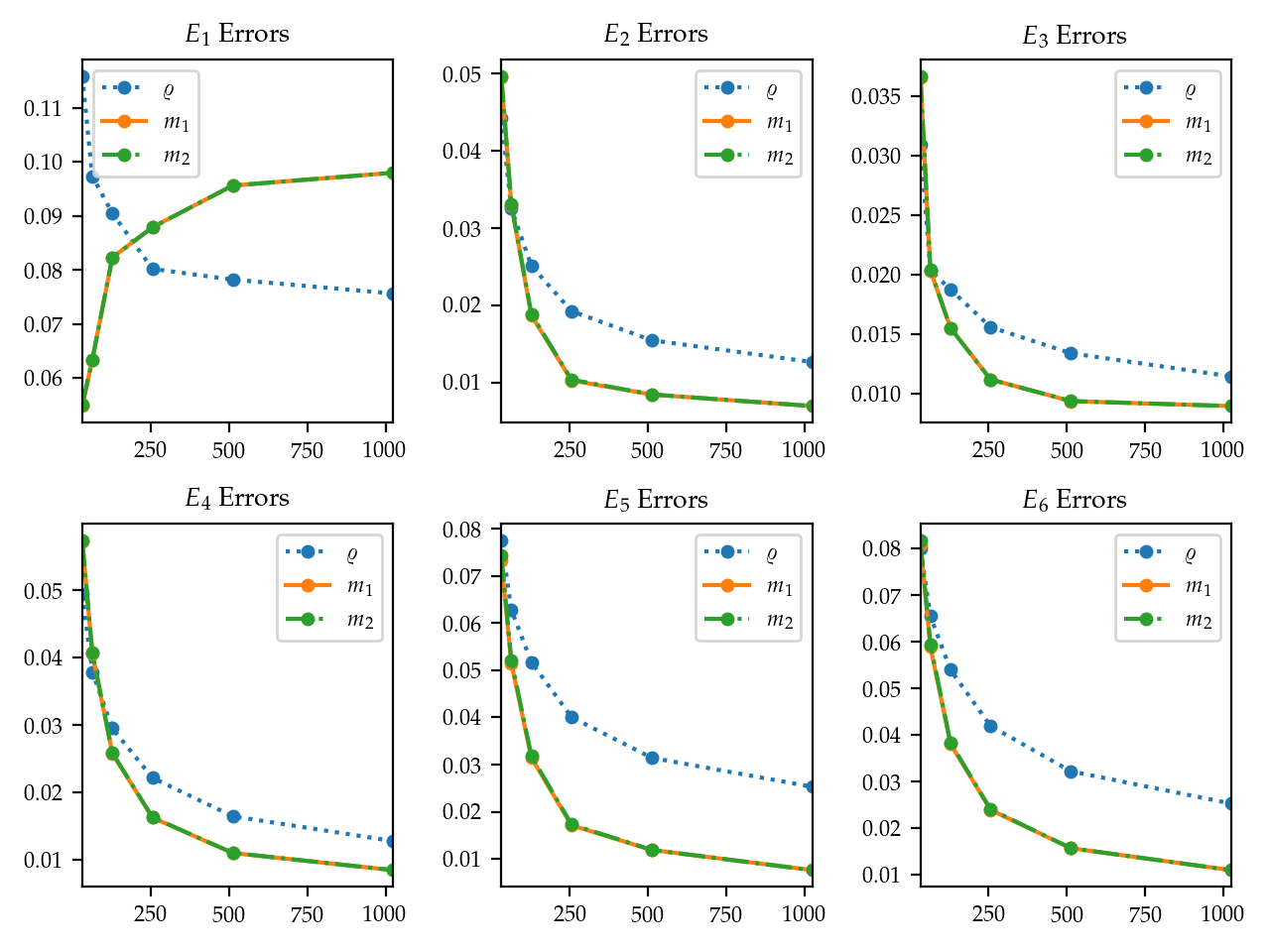}
    \caption{Error profiles for the cylindrical explosion problem.}
    \label{fig:cyl-exp-err}
\end{figure}

\begin{table}[htpb]
    {\begin{flushleft}
    \begin{tabular}{|c|c|c|c|c|c|c|c|c|c|c|c|c|}
    \hline
        & \multicolumn{2}{|c|}{$E_1$} & \multicolumn{2}{|c|}{$E_2$} & \multicolumn{2}{|c|}{$E_3$} & \multicolumn{2}{|c|}{$E_4$} & \multicolumn{2}{|c|}{$E_5$} & \multicolumn{2}{|c|}{$E_6$}\\ \cline{2-13}
        $k$ & error & order & error & order & error & order & error & order & error & order & error & order \\  
        \hline
        32 & 1.15e-1 & - & 4.42e-2 & - & 3.09e-2 & - & 4.94e-2 & - & 7.77e-2 & - & 7.99e-2 & - \\
        \hline
        64 & 9.73e-2 & 0.25 & 3.25e-2 & 0.44 & 2.03e-2 & 0.60  & 3.78e-2 & 0.38 & 6.27e-2 & 0.30 & 6.54e-2 & 0.28 \\
        \hline
        128 & 9.05e-2 & 0.10 & 2.51e-2 & 0.37 & 1.87e-2 & 0.11  & 2.95e-2 &  0.35 & 5.17e-2 & 0.28 & 5.39e-2 & 0.27 \\
        \hline
        256 & 8.02e-2 & 0.17 & 1.92e-2 & 0.38 & 1.56e-2 & 0.26  & 2.21e-2 & 0.41 & 4.01e-2 & 0.36 & 4.19e-2 & 0.36 \\
        \hline
        512 & 7.82e-2 & 0.03 & 1.54e-2 & 0.31 & 1.34e-2 & 0.21 & 1.64e-2 & 0.42 & 3.14e-2 & 0.34 & 3.22e-2 & 0.37 \\
        \hline
        1024 & 7.56e-2 & 0.04 & 1.26e-2 & 0.29 & 1.14e-2 & 0.22 & 1.27e-2 & 0.36 & 2.53e-2 & 0.31 & 2.53e-2 & 0.34\\
        \hline
    \end{tabular}
    \caption{Errors and convergence rates of $\vrho$ for the cylindrical explosion problem}
    \label{tab:ord-den-cyl-exp}
    \end{flushleft}
    }
\end{table}

\begin{remark}
    \label{rem:cyl-exp-yg-meas}
    For the above problem, the $E_1$ error in the case of density is decreasing, indicating the strong convergence of the sequence of numerical solutions $\lbrace\vrho^\m\rbrace_{m\in\N}$ in $L^1$. In accordance with \cite[Proposition 6.13]{Ped97}, the measures $\V$ generated by the sequence of numerical solutions of the above problem can be decomposed as $\V = \delta_{\vrho(t,x)}\otimes \lambda_{t,x}$ for a.e.\ $(t,x)\in\odom$, where $\vrho$ denotes the $L^1$-limit of the sequence $\lbrace\vrho^\m\rbrace_{m\in\N}$ and $\lbrace\lambda_{t,x}\rbrace_{(t,x)\in\odom}$ is the Young measure generated by the sequence $\lbrace \bm^\m\rbrace_{m\in\N}$. 
\end{remark}

\subsection{Kelvin-Helmholtz Instability}
\label{subsec:kel_helm}

We consider the classical problem of Kelvin-Helmholtz instability. The set up is same as that of \cite{Pri08}. The computational domain is set as $\lbrack -0.5, 0.5\rbrack\times\lbrack -0.5,0.5\rbrack$ with periodic boundaries everywhere. The initial density reads 
\begin{equation*}
    \vrho(0,x_1,x_2) = 
    \begin{dcases}
        2, &\text{if }\abs{x_2}<0.25, \\
        1 ,&\text{otherwise},
    \end{dcases}
\end{equation*}
and the pressure law is given by $p(\vrho) = \vrho^{5/3}$. A cartesian shear flow is set up in the $x_1$-direction with velocity 
\begin{equation*}
    u_1(0,x_1,x_2) = 
    \begin{dcases}
        -0.5, &\text{if }\abs{x_2}<0.25, \\
        0.5 ,&\text{otherwise}.
    \end{dcases}
\end{equation*}
Such a configuration is known to be unstable at all wavelengths. We seed the instability by applying a small velocity perturbation in the $x_2$-direction given by 
\begin{equation*}
    u_2(0,x_1,x_2) = 
    \begin{dcases}
        A\sin\lbrack -2\pi(x_1+0.5)/\lambda\rbrack, &\text{if }\abs{x_2-0.25}<0.025, \\
        A\sin\lbrack 2\pi(x_1+0.5)/\lambda\rbrack, &\text{if }\abs{x_2+0.25}<0.025,
    \end{dcases}
\end{equation*}
where we set $A = 0.025$ and $\lambda = 1/6$. For the chosen density ratio of $2:1$, the time after which the instabilities will start to develop is around $t = 0.35$; see \cite{Pri08} for details. Hence, we set the final time as $T=0.4$ and also set $K = 1024$. 

In Figure \ref{fig:den-kel-helm}, we present the density profiles for successive mesh refinements and we can see that the instabilities are just starting to develop when $k = 512$. In Figure \ref{fig:kel-helm-err}, we present the error plots for the density $\vrho$ and the momentum $\bm = (m_1, m_2)$. The numerical results clearly show that all the errors are decreasing upon refining the mesh, leading us to conclude that the Young measure generated by the numerical solution is the family of Dirac masses centred at the limit solution. We also compute the $L^1$-norms of the relative entropy of the numerical solutions with respect to the reference solution, for each $k$, in order to determine whether the limit solution is a strong solution or a weak solution of the Euler system. The results are presented in Table \ref{tab:rel-ent-kel-helm} and we observe that the norm of the relative entropies are also decreasing, indicating that the limit is a strong solution of the Euler system in accordance with the weak-strong uniqueness principle; cf.\ Remark \ref{rem:rel-ent} and Theorem \ref{thm:wk-str-uniq}. The maximum value of $\eta\sim 1.76$ for this problem.

\begin{figure}[htpb]
    \centering
    \includegraphics[height = 0.15\textheight]{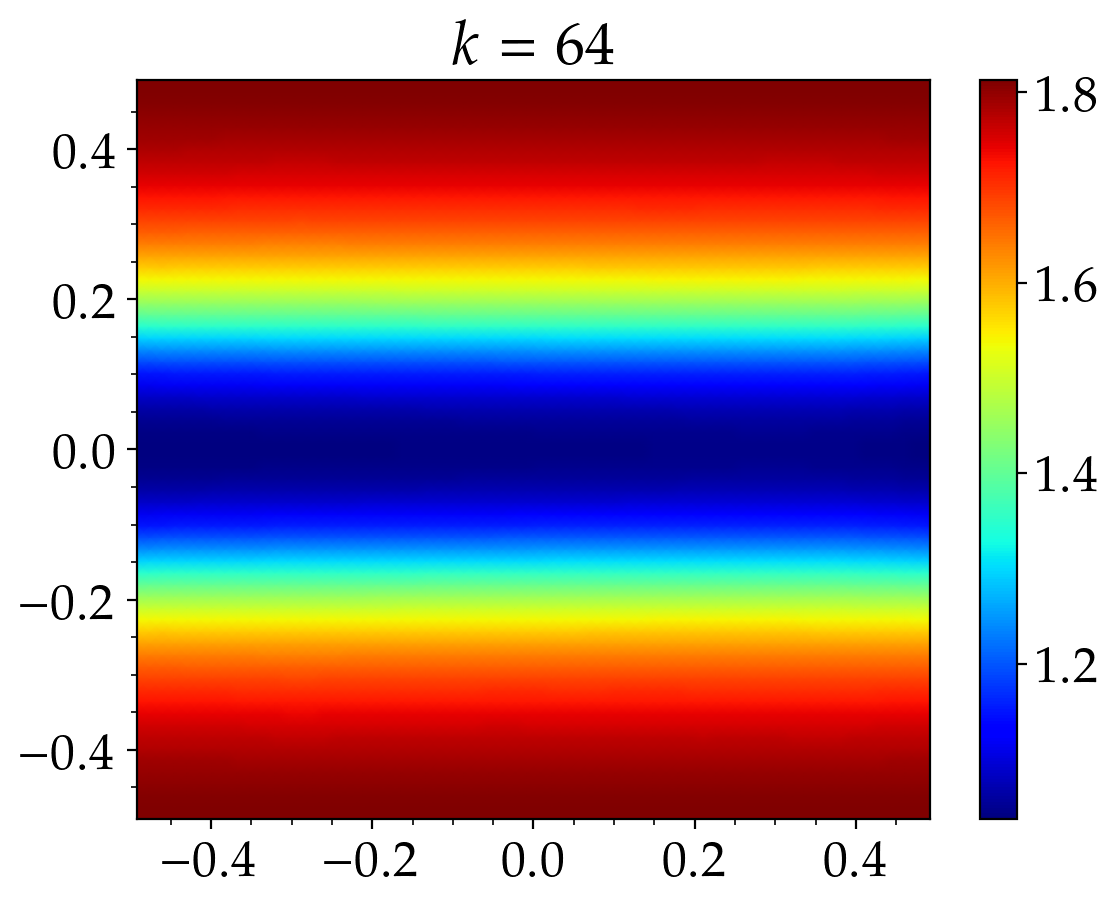}
    \includegraphics[height = 0.15\textheight]{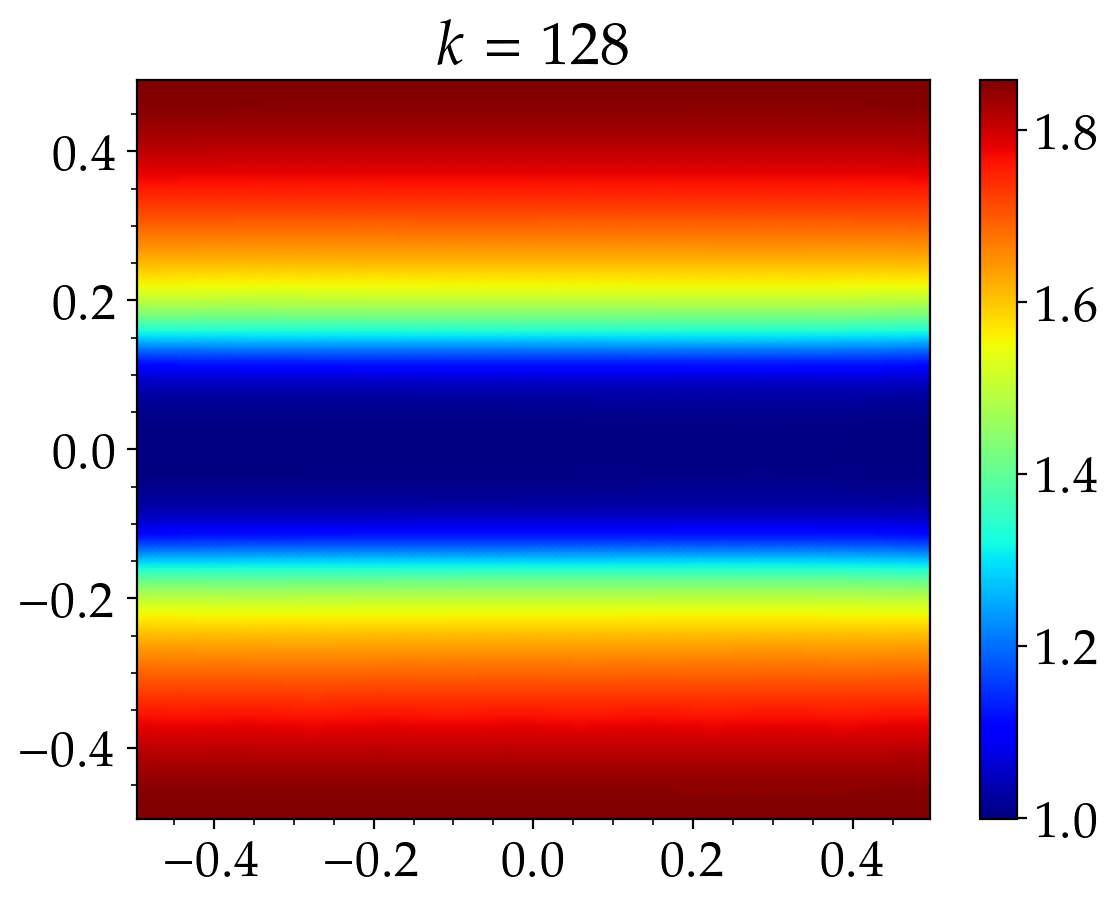}
    \includegraphics[height = 0.15\textheight]{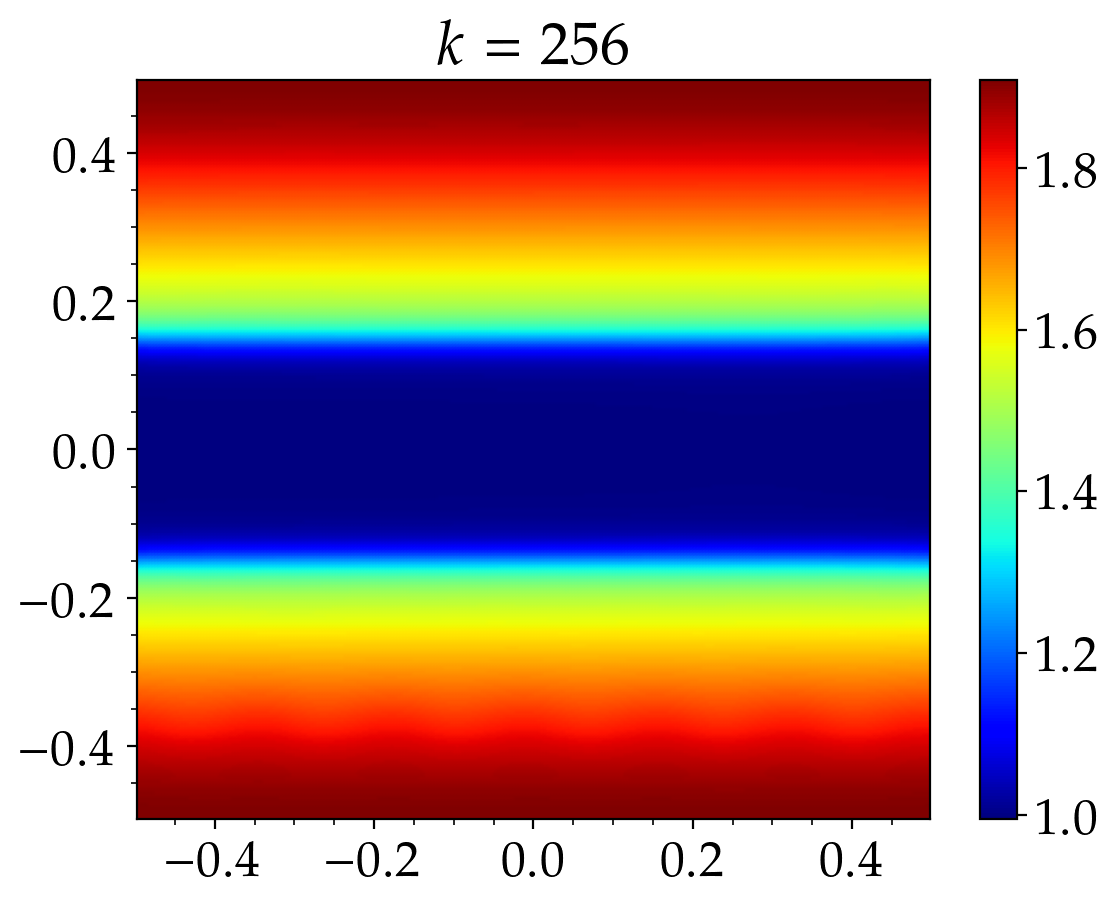}
    \includegraphics[height = 0.15\textheight]{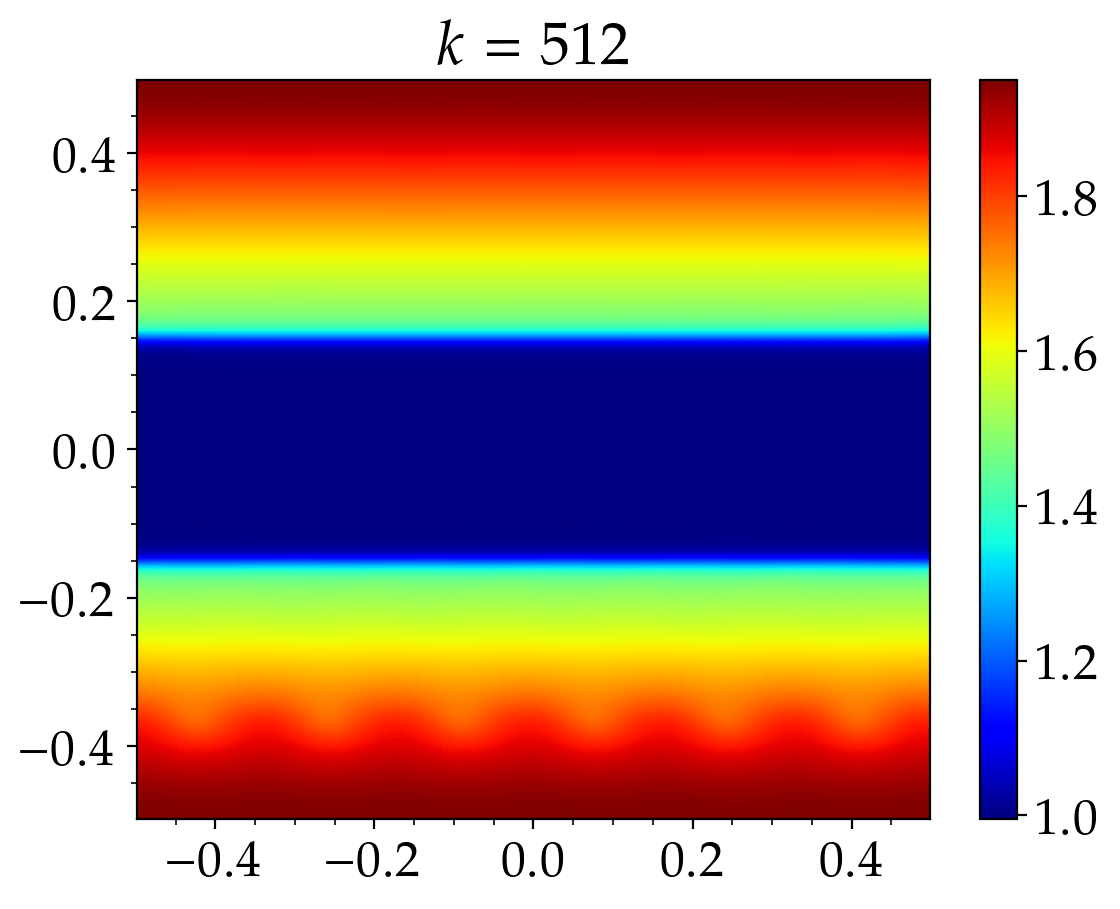}
    \caption{Density profiles of the Kelvin-Helmholtz problem for different mesh sizes.}
    \label{fig:den-kel-helm}
\end{figure}

\begin{figure}[htpb]
    \centering
    \includegraphics[height = 0.4\textheight]{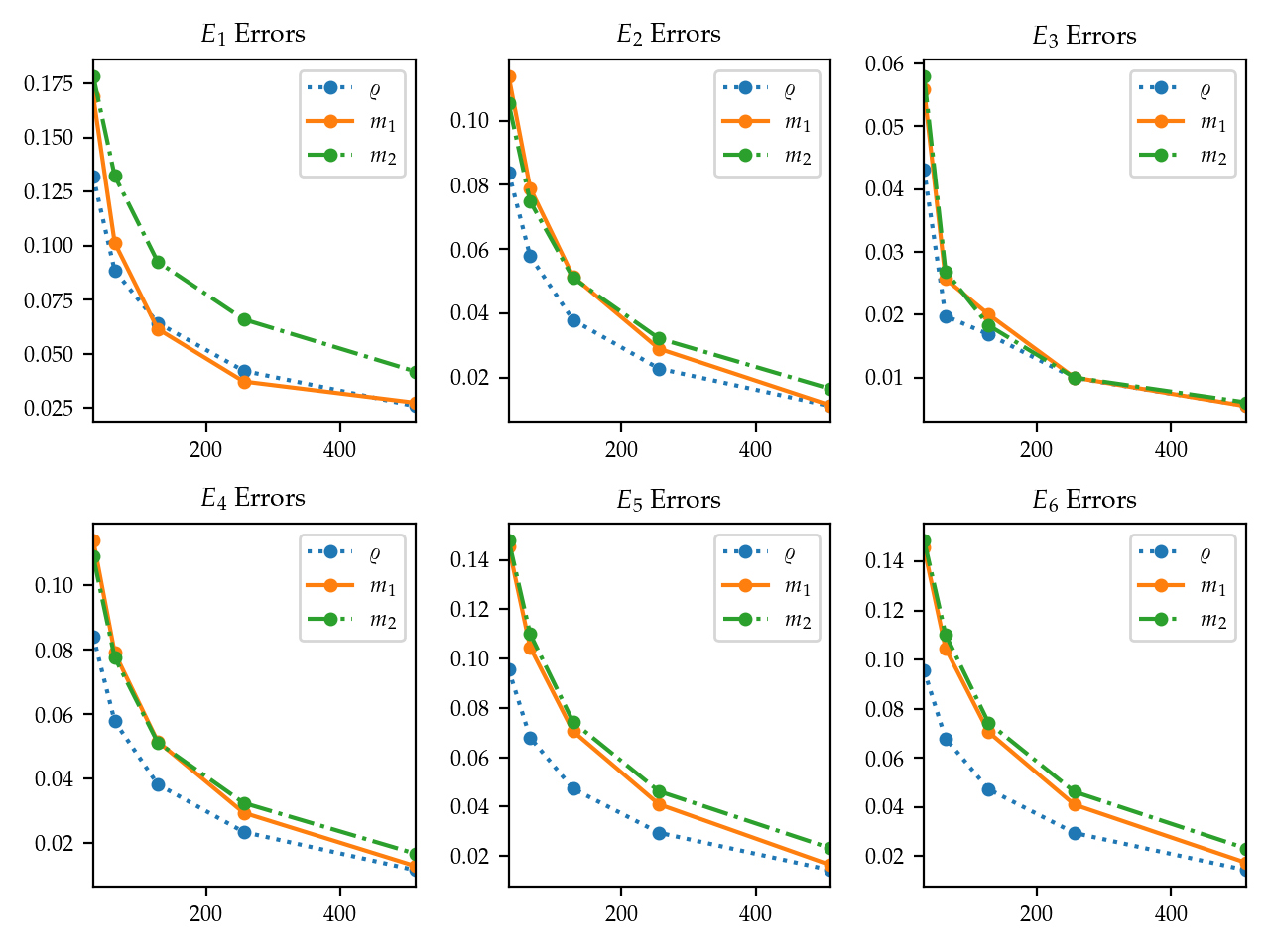}
    \caption{Error profiles of the Kelvin-Helmholtz problem.}
    \label{fig:kel-helm-err}
\end{figure}

\begin{table}[htpb]
    \centering
    \begin{tabular}{|c|c|c|c|c|c|}
          \hline
          $k$ & 32 & 64 & 128 & 256 & 512 \\
          \hline
          Relative Entropy & 0.0463 & 0.0245 & 0.0123 & 0.0059 & 0.0028 \\
          \hline
    \end{tabular}
    \caption{Relative entropy values for the Kelvin-Helmholtz problem for different mesh sizes.}
    \label{tab:rel-ent-kel-helm}
\end{table}

\subsection{\texorpdfstring{$\delta$}{\textdelta}-Shockwave}
\label{subsec:del-shock}

We consider the following $1D$ Riemann problem from \cite{CL03, Jeg18}, where the domain is set as $\lbrack -1,1\rbrack$ and the initial states are given by
\begin{equation*}
    (\vrho, u)(0,x) = 
    \begin{dcases}
        (1, 1.5), &\text{if }x<0, \\
        (0.2, 0), &\text{if }x>0.
    \end{dcases}
\end{equation*}
The pressure law is taken as $p(\vrho) = \kappa^2\vrho^{1.4},\ \kappa>0$. It has been shown in \cite{CL03} that as $\kappa\to 0$, the isentropic Euler system with the above pressure law reduces to the zero pressure equations of sticky particles. The zero pressure model admits vacuum states and $\delta$-shockwaves due to coincident and resonant characteristic fields. Consequently, any two-shock Riemann solution tends to a $\delta$-shock of the zero pressure Euler equations and the density between the two shocks tends to a weighted Dirac-measure which in turn forms the $\delta$-shock; see \cite{CL03} for more details. We compute the numerical solutions for $\kappa = 1, 10^{-2}, 10^{-3}, 10^{-5}$, with final time $T=0.2$ and $K = 2048$.

For the sake of brevity, we present the plots of the density when $k = 1024$ , for the cases of $\kappa = 1$ and $\kappa = 10^{-5}$ in Figure \ref{fig:del_shock_den}. Clearly, we see the formation of classical shock wave corresponding to $\kappa=1$, whereas the figure for $\kappa=10^{-5}$ indicates the formation of vacuum states and a $\delta$-shock. We did not observe any instability or the breaking down since the present scheme is positivity preserving. Furthermore, the proposed scheme is able to effectively capture non-classical $\delta$-shock type discontinuities. 

In Figures \ref{fig:err-kappa-1}, \ref{fig:err-kappa-10^-2}, \ref{fig:err-kappa-10^-3} and \ref{fig:err-kappa-10^-5}, we present the error plots of the density $\vrho$ and the momentum $m$, for $\kappa=1, 10^{-2}, 10^{-3}$ and $10^{-5}$ respectively. For each $\kappa$, all the errors decrease upon refining the mesh, indicating that the limit of the numerical solutions for each case is a weak solution of the Euler system. Therefore, we also compute the $L^1$-norms of the relative entropy in each case to check whether convergence towards a strong solution of the system is achieved or not. The results are presented in Figure \ref{fig:rel-ent-del-shock} and we observe that for $\kappa = 1$, the limit is indeed a strong solution of the Euler system, cf.\ Remark \ref{rem:rel-ent} and Theorem \ref{thm:wk-str-uniq}. For the cases of $\kappa = 10^{-2}, 10^{-3}$ and $10^{-5}$, it seems that we have no convergence towards a strong solution due to the disordered nature of the plots, which is in line with the theoretical expectations. The maximum value of $\eta$ observed was $\eta\sim 7.5,\,93.05,\,98.18,\,98.29$ for $\kappa = 1,\,10^{-2},\,10^{-3},\,10^{-5}$ respectively.

\begin{figure}[htpb]
    \centering
        \includegraphics[width = 0.4\textwidth]{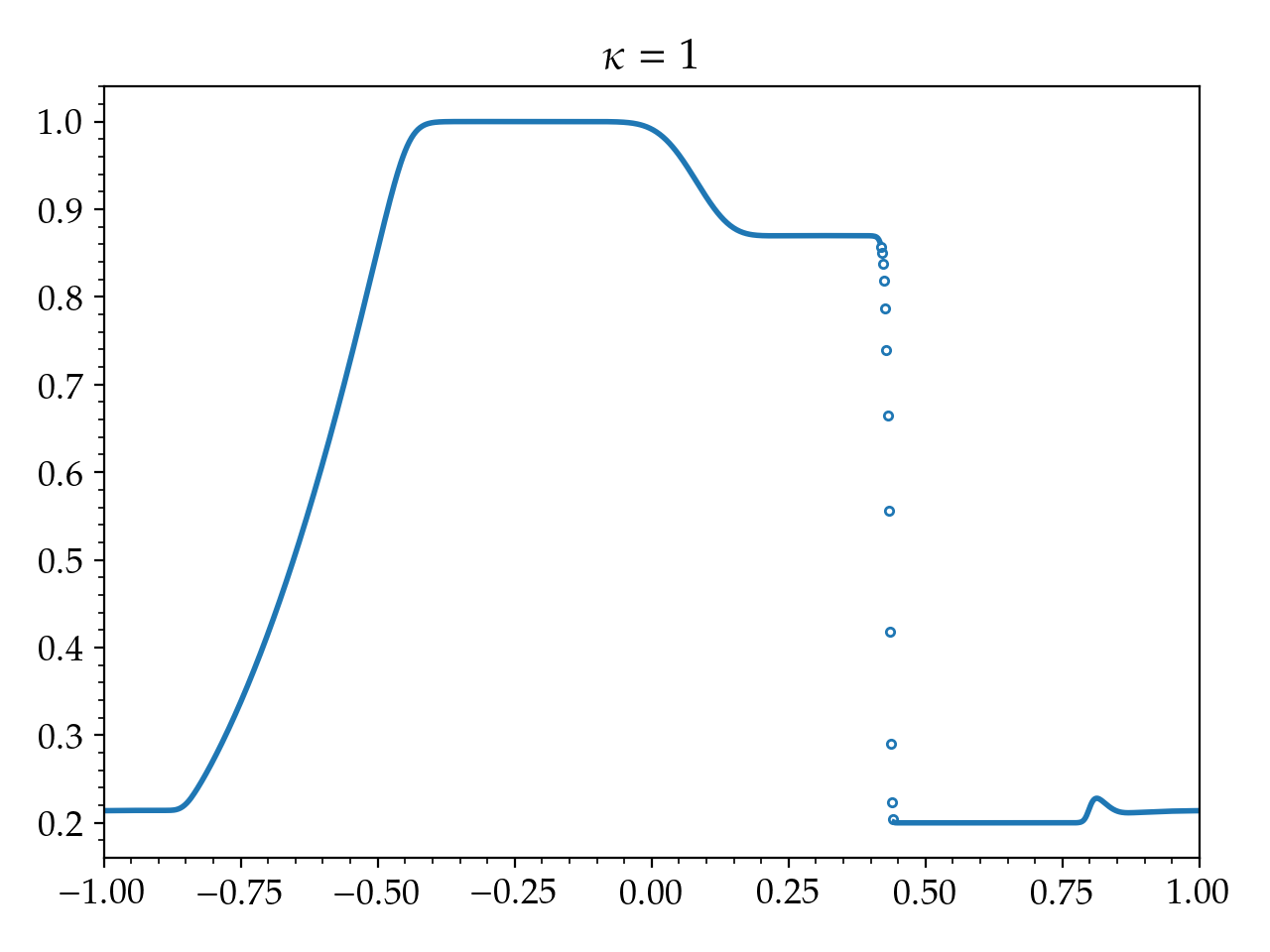}
        \includegraphics[width = 0.4\textwidth]{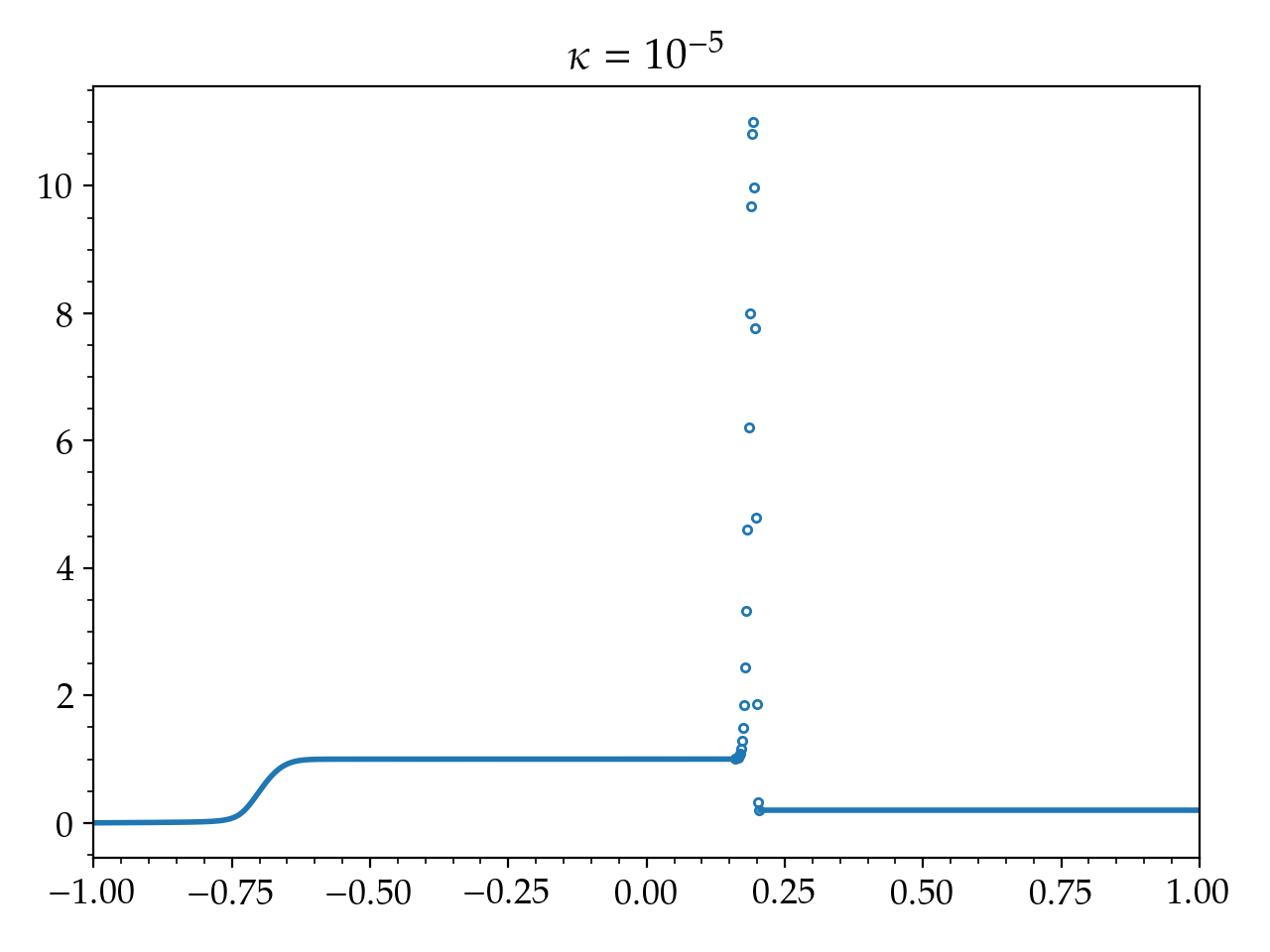}
    \caption{Density profiles of the $\delta$-shockwave problem for $k = 1024$.}
    \label{fig:del_shock_den}
\end{figure}

\begin{figure}[htpb]
    \centering
    \includegraphics[height = 0.45\textheight]{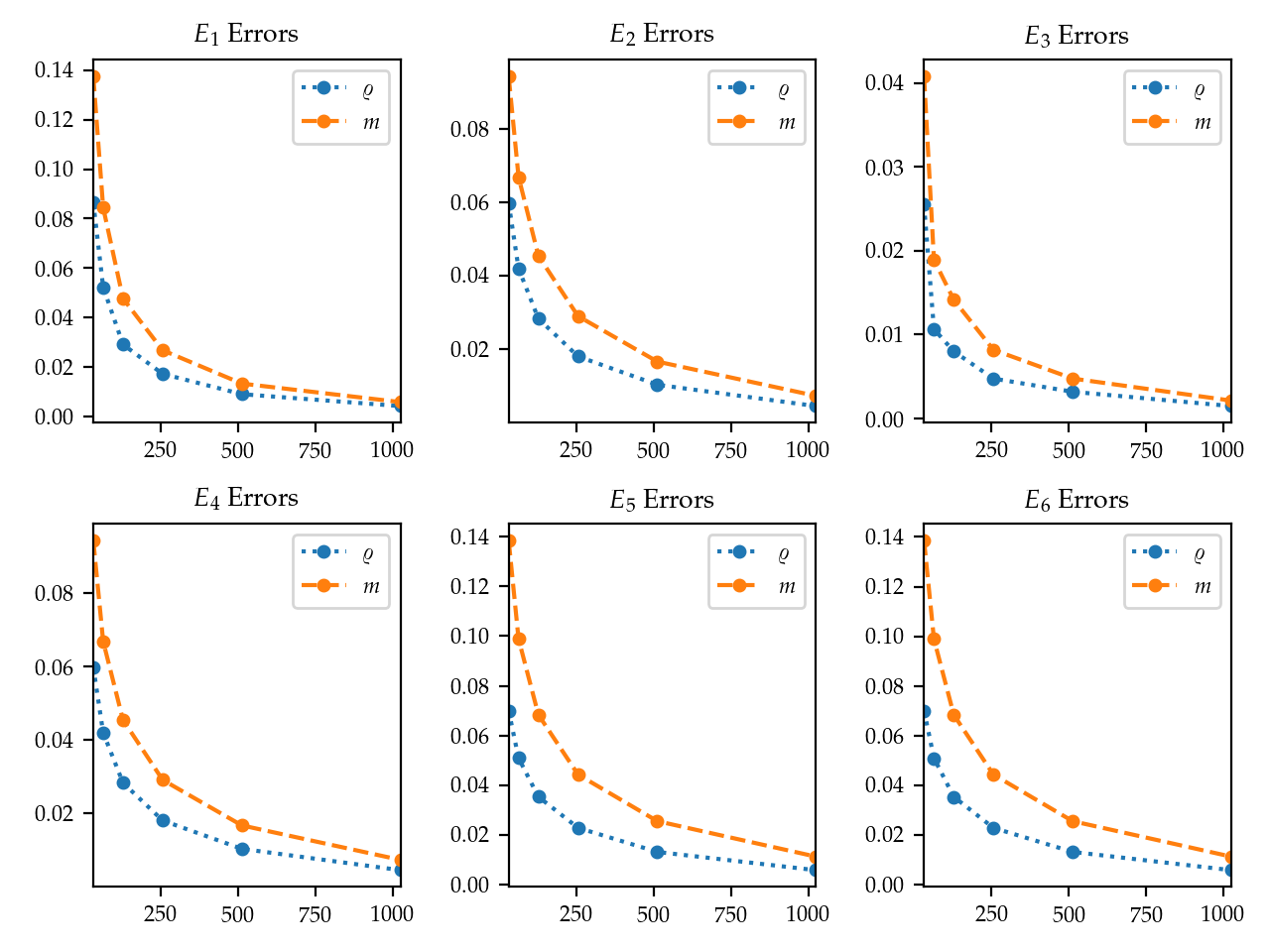}
    \caption{Convergence analysis for $\kappa = 1$.}
    \label{fig:err-kappa-1}
\end{figure}

\begin{figure}[htpb]
    \centering
    \includegraphics[height = 0.45\textheight]{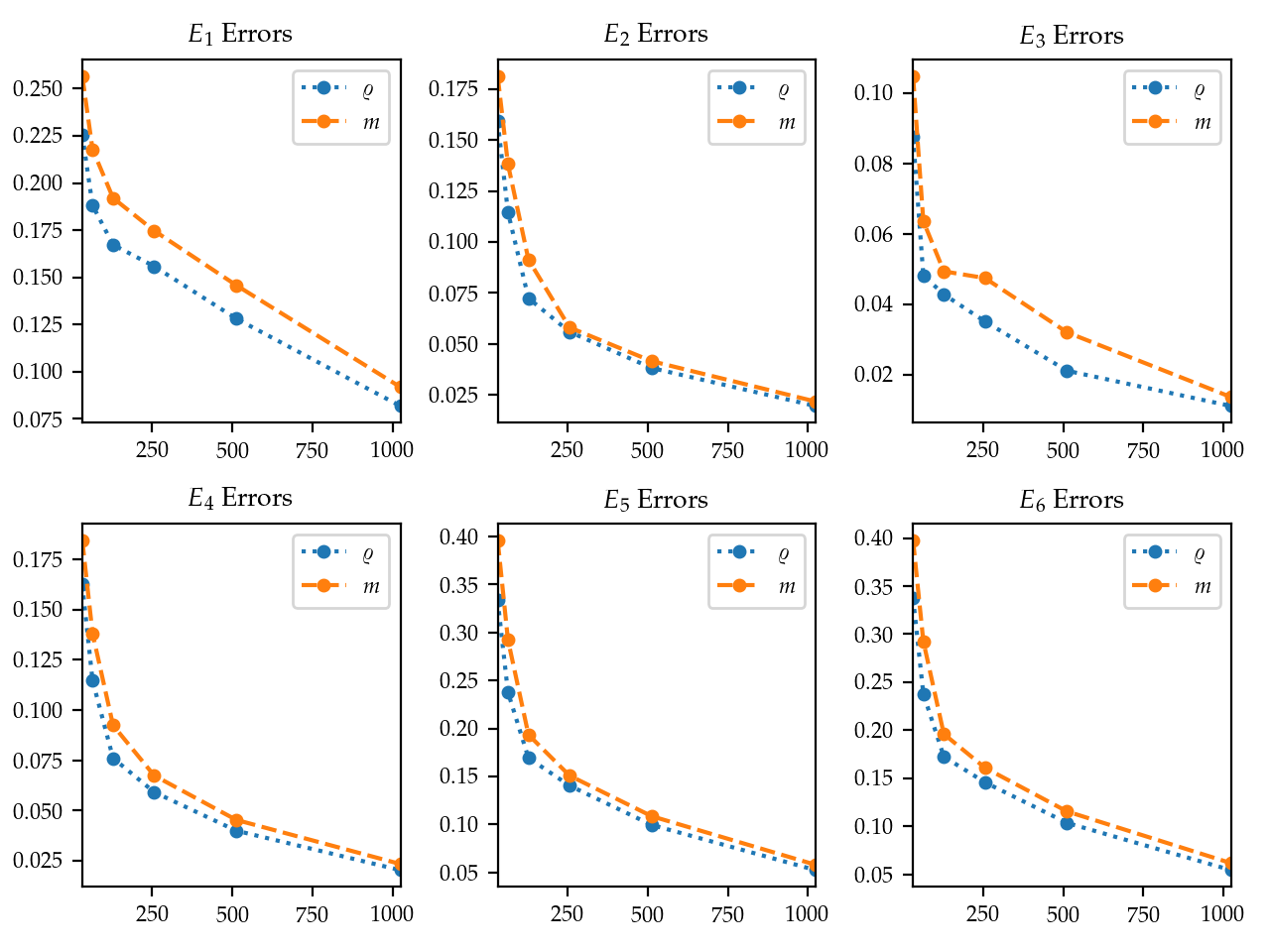}
    \caption{Convergence analysis for $\kappa = 10^{-2}$.}
    \label{fig:err-kappa-10^-2}
\end{figure}

\begin{figure}[htpb]
    \centering
    \includegraphics[height = 0.45\textheight]{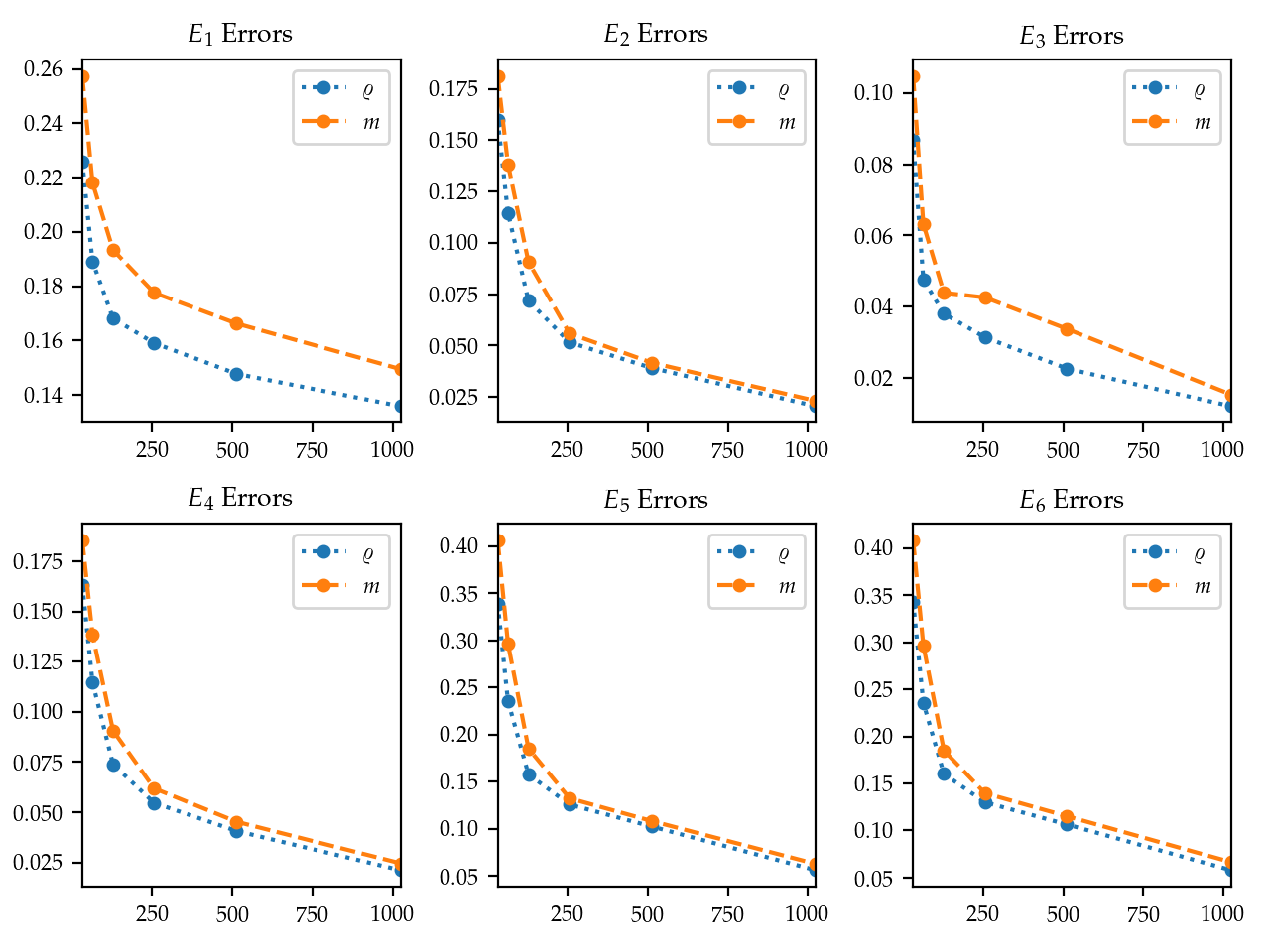}
    \caption{Convergence analysis for $\kappa = 10^{-3}$.}
    \label{fig:err-kappa-10^-3}
\end{figure}

\begin{figure}[htpb]
    \centering
    \includegraphics[height = 0.45\textheight]{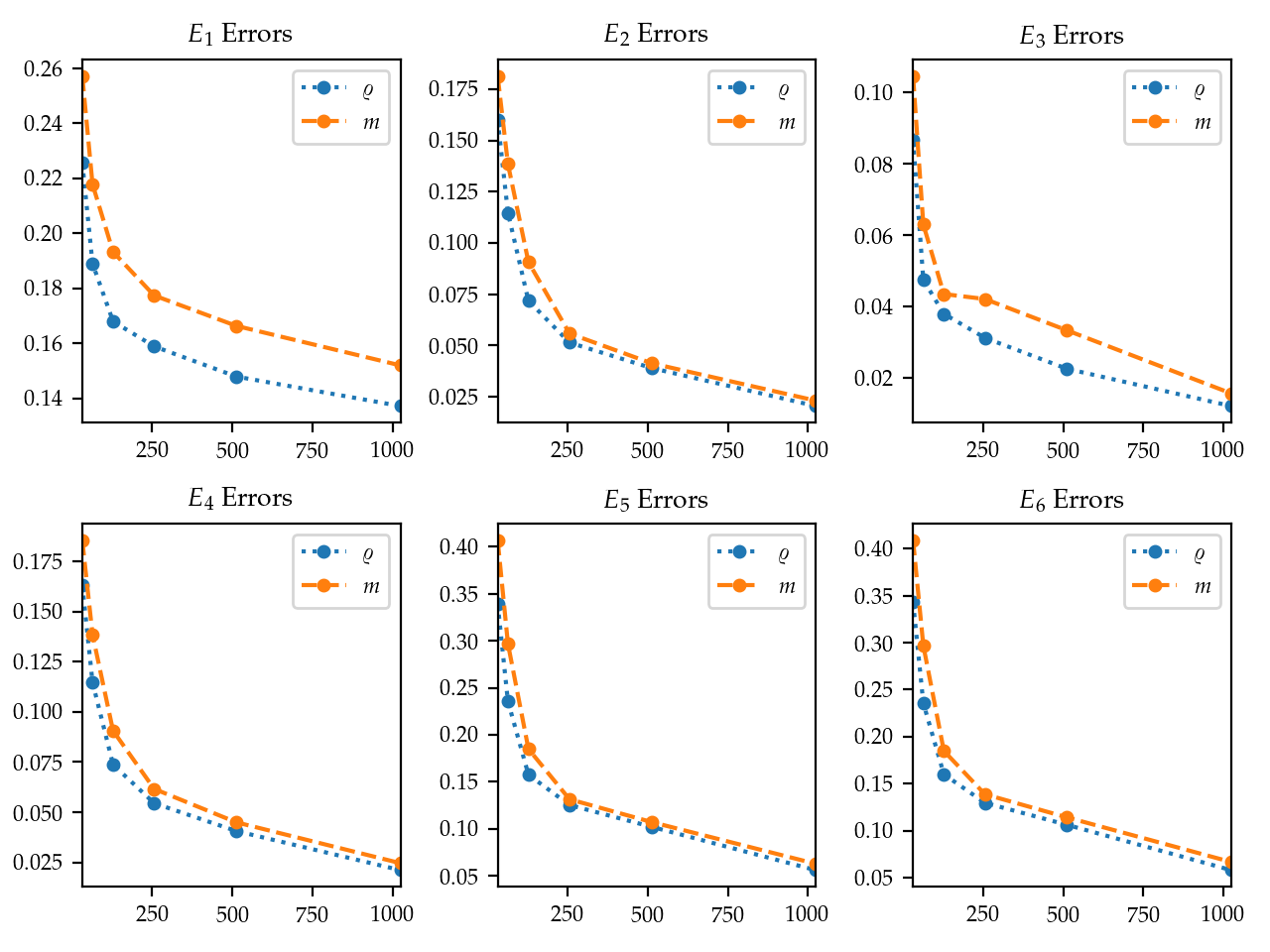}
    \caption{Convergence analysis for $\kappa = 10^{-5}$.}
    \label{fig:err-kappa-10^-5}
\end{figure}

\begin{figure}[htpb]
    \centering
    \includegraphics[height = 0.45\textheight]{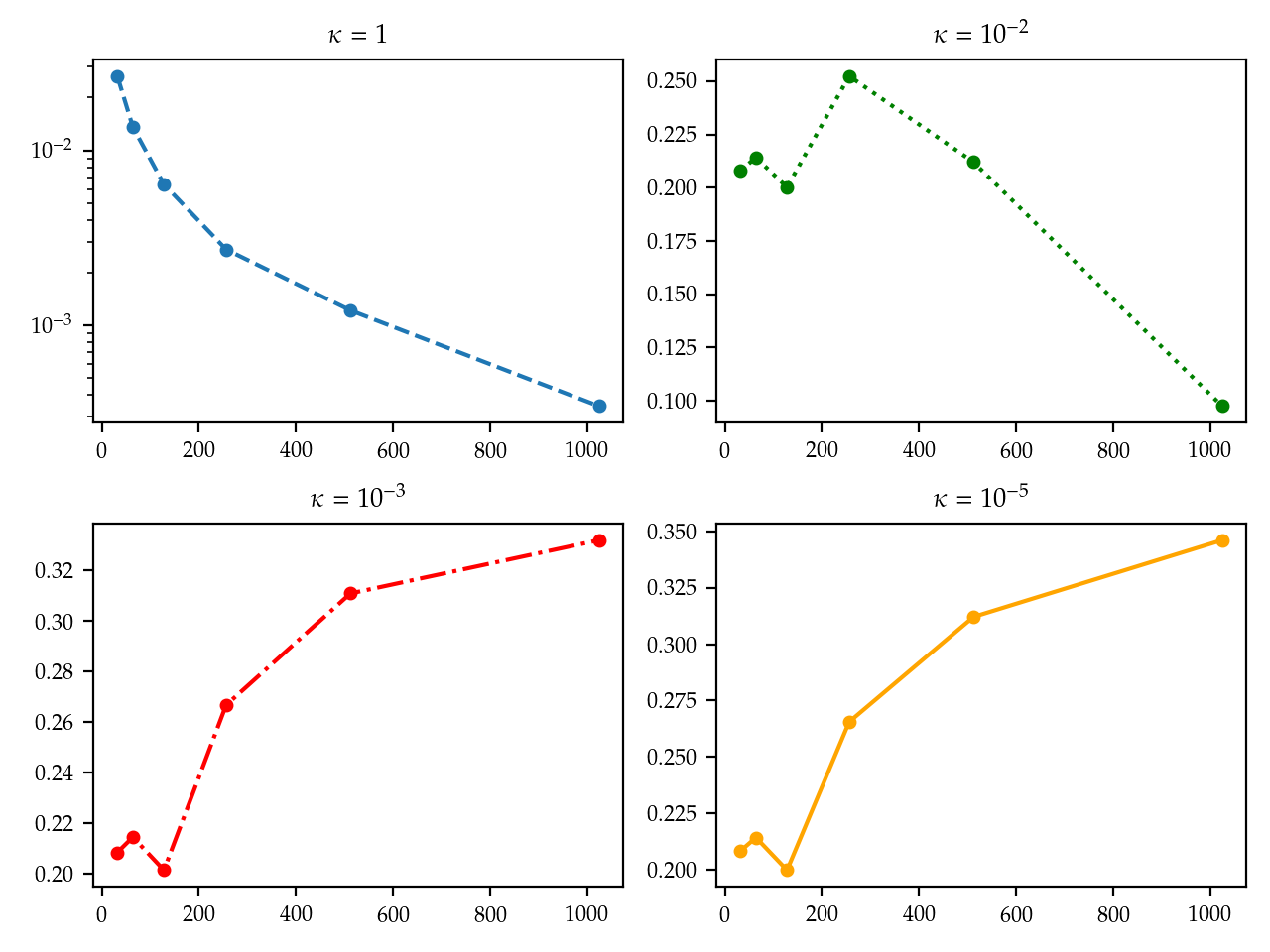}
    \caption{Relative entropy plots for the $\delta$-shockwave problem for different values of $\kappa$.}
    \label{fig:rel-ent-del-shock}
\end{figure}

\begin{remark}
    We note that in the above problem, corresponding to $\kappa=10^{-5}$, we observe the violation of theoretical estimates assumed in Theorem~\ref{thm:cons-scheme}, i.e.\ related to the boundedness of the numerical density. This indicates that the present scheme is robust and that it can perform in the degenerate cases where vacuum is formed and the density is unbounded above.    
\end{remark}

\section{Concluding Remarks}

A semi-implicit in time, entropy stable finite volume scheme for the barotropic Euler system is designed, analyzed and implemented. The entropy stability is achieved  via the introduction of a shifted velocity proportional to the pressure gradient in the convective fluxes of the mass and momentum balances, under the assumption of a CFL-like condition to ensure stability. The proposed scheme possesses certain key physical features, including the positivity of the density along with the aforementioned entropy stability, classifying the proposed scheme as an `invariant-domain' or a `structure-preserving' method. It is rigorously shown that the numerical solutions generated by the scheme are consistent with the weak formulation of the barotropic Euler system modulo some perturbation terms that vanish in the limit, under some physically reasonable boundedness assumptions. The weak convergence of the numerical solutions to a DMV solution then follows, and further results pertaining to the strong convergence of the Ces\`{a}ro averages of the numerical solutions are deduced using the tool of $\K$-convergence. The results of the numerical case studies clearly substantiate the claims made, wherein we exhibit the convergence of the scheme to a DMV solution, a strong solution and a weak solution of the Euler system. Further, the numerical results also highlight the fact that the scheme is robust, it can perform even in degenerate cases where vacuum is formed and density is unbounded above and it can also capture non-classical shocks such as $\delta$-shocks. As far as the authors are aware, this is the first result in literature where a semi-implicit setup is used in designing a scheme that generates DMV solutions of the Euler system.

\section*{Acknowledgements}
The authors thank Mainak Kar for valuable discussions on the topic and his help in numerical implementation, and the anonymous reviewers whose comments have helped improve the manuscript.

\section*{Appendix}
\label{sec:appdx}
In this section, we present the key steps that yield the discrete variants of the renormalization identity and the kinetic energy identity, namely the equations \eqref{eqn:disc-renorm} and \eqref{eqn:disc-kin} in Theorem \ref{thm:disc-enrg-est}. 

\subsection*{A1:\ Discrete Renormalization Identity}
\label{subsec:A1}

Multiply the discrete mass balance \eqref{eqn:disc-mss-bal} with $\psig^\prime(\vrho^{n+1}_K)$ to yield $\mathcal{A}_1 + \mathcal{A}_2 = 0$, where 
\begin{align*}
    &\mathcal{A}_1 = \frac{\psig^\prime(\rk^{n+1})}{\delt}(\rk^{n+1} - \rk^n), \\
    &\mathcal{A}_2 = \frac{1}{\absk}\sum_{\substack{\sink \\ \sigma = K\vert L}} \psig^\prime(\rk^{n+1})\flx(\vrho^{n+1}, \bv^n).
\end{align*}

To simplify $\mathcal{A}_1$, we employ a Taylor expansion to get 
\begin{equation}
\label{eqn:app_1}
    \mathcal{A}_1 = \frac{1}{\delt}(\psig(\rk^{n+1}) - \psig(\rk^n)) + \frac{1}{2\delt}(\rk^{n+1} - \rk^n)^2\psig^{\prime\prime}(\overline{\vrho}^{n+1/2}_K),
\end{equation}
where $\overline{\vrho}_K^{n+1/2}\in\llbracket \rk^n,\rk^{n+1}\rrbracket$.

Next, using the relation $\vrho^{n+1}_K\psig^\prime(\rk^{n+1}) - p^{n+1}_K = \psig(\rk^{n+1})$, cf.\ \eqref{eqn:pres-pot-de}, enables us to write 
\begin{equation}
\label{eqn:app_2}
    \begin{split}
        \psig^\prime(\rk^{n+1})\flx(\vrho^{n+1}, \bv^n) &= \abssig \Bigl(p^{n+1}_K(v^{n}_{\sk})^{+} + \psig(\rk^{n+1})(v^{n}_{\sk})^{+} + \psig^\prime(\vrho^{n+1}_K)(\vrho^{n+1}_L - \rk^{n+1})(v^n_{\sk})^{-} \\
        &+ \rk^{n+1}\psig^{\prime}(\vrho^{n+1}_K)(v^n_{\sk})^{-}\Bigr).
    \end{split}
\end{equation}

Again, employing a Taylor expansion to rewrite the second term in \eqref{eqn:app_2} and simplifying the resulting expression yields the following form for $\mathcal{A}_2$: 

\begin{equation}
\label{eqn:app_3}
    \mathcal{A}_2 = (\divup(\psig(\vrho^{n+1}), \bv^n))_K + p^{n+1}_K(\divt\bv^n)_K + \frac{1}{2\absk}\sum_{\substack{\sink \\ K\vert L}}\abssig(-v^n_{\sk})^{-}(\vrho^{n+1}_L - \vrho^{n+1}_K)^2\psig^{\prime\prime}(\tilde{\vrho}^{n+1}_\sigma),
\end{equation}
where $\tilde{\vrho}^{n+1}_\sigma \in \llbracket\rk^{n+1}, \vrho^{n+1}_L\rrbracket$. Finally, summing up \eqref{eqn:app_1} and \eqref{eqn:app_3} yields the renormalization identity \eqref{eqn:disc-renorm}.

\subsection*{A2:\ Discrete Kinetic Energy Identity}
\label{subsec:A2}

We consider the velocity update \eqref{eqn:disc-vel-upd} which reads
\[
\frac{1}{\delt}(\uk^{n+1}-\uk^n)+\frac{1}{\absk}\sum_{\substack{\sink \\ \sigma=K\vert L}}\abssig\Biggl(\frac{\bu^n_L-\uk^n}{\rk^{n+1}}\Biggr) \flx^{n+1,-}+\frac{1}{\rk^{n+1}}(\gradt p^{n+1})_K=0.
\]

Taking the dot product of the above equation with $\rk^{n+1}\uk^n$ and using the algebraic identity $(a-b)\cdot b = (\abs{a}^2 - \abs{b}^2 - \abs{a-b}^2)/2$ in the first and second terms yields 
\begin{equation}
\label{eqn:app_4}
    \begin{split}
        \frac{\rk^{n+1}}{\delt}&\biggl(\half\abs{\uk^{n+1}}^2 - \half\abs{\uk^n}^2 - \half\abs{\uk^{n+1} - \uk^n}^2\biggr) \\
        &+ \frac{1}{\absk}\sum_{\substack{\sink \\ \sigma = K\vert L}}\abssig\biggl(\half\abs{\bu^n_L}^2 - \half\abs{\uk^n}^2 - \half\abs{\bu^n_L - \uk^n}^2\biggr)F^{n+1, -}_{\sk} + (\gradt p^{n+1})_K\cdot\bv_K^n = -\eta\delt\abs{(\gradt p^{n+1})_K}^2.
    \end{split}
\end{equation}
Here, we have used $\uk^n = \bv^n_K + \eta\delt(\gradt p^{n+1})_K$ to rewrite the pressure gradient term.
Multiplying the discrete mass balance \eqref{eqn:disc-mss-bal} by $\abs{\uk^n}^2/2$, summing it up with \eqref{eqn:app_4} and simplifying the resulting expression gives us the kinetic energy identity \eqref{eqn:disc-kin}.

\bibliographystyle{abbrv}
\bibliography{ref}
\end{document}